\documentclass[reqno,11pt, psfig]{amsart}
\usepackage{cases}
\usepackage{mathrsfs}
\usepackage[T1]{fontenc}
\usepackage{amsmath,latexsym,amssymb,amsfonts,amsbsy, amsthm}
\usepackage{bm}
\usepackage[usenames]{color}
\usepackage{xspace,colortbl}
\usepackage{epsfig}
\usepackage{graphicx}
\usepackage{subfigure}
\usepackage{amsmath,amsfonts,amsthm,amssymb,amscd}
\input amssym.def
\input amssym.tex
\usepackage{color}
\usepackage[title,titletoc,toc]{appendix}
\usepackage{bigints}

\allowdisplaybreaks

\newcommand{\be}{\begin{equation} }
	\newcommand{\ee}{\end{equation}}
\newcommand{\bee}{\begin{equation*} }
	\newcommand{\eee}{\end{equation*}}
\newcommand{\bse}{\begin{subequations}}
	\newcommand{\ese}{\end{subequations}}

\DeclareMathOperator*{\im}{Im}
\DeclareMathOperator*{\re}{Re}
\newcommand*{\dif}{\mathop{}\!\mathrm{d}}

\usepackage {caption}
\usepackage{xcolor}
\usepackage{cases}

\usepackage{amsmath}
\usepackage{amsthm}

\usepackage{todonotes}
\usepackage{float}
\newtheorem{theorem}{Theorem}[section]
\newtheorem{corollary}[theorem]{Corollary}

\newtheorem{lemma}{Lemma}[section]
\newtheorem{Proposition}{Proposition}[section]

\theoremstyle{remark}
\newtheorem{remark}{Remark}[section]

\theoremstyle{definition}

\numberwithin{equation}{section}

\title[$L^2$-Orbital Stability for DNLS equation]
{Orbital Stability of Soliton for the Derivative Nonlinear Schr\"odinger Equation in the $L^2$ Space}

\author[]{Yiling YANG}
\address[Yiling Yang]{School of Mathematical Sciences  and Key Laboratory   for Nonlinear Science, Fudan   University,  P. R. China.}\email{ylyang19@fudan.edu.cn}
\author[]{Engui Fan}
\address[Engui Fan]{School of Mathematical Sciences  and Key Laboratory  for Nonlinear Science, Fudan   University,  P. R. China.}\email{faneg@fudan.edu.cn}
\author[]{Yue Liu}
\address[Yue Liu]{Department of Mathematics, University of Texas at Arlington, TX 76019}\email{yliu@uta.edu}

\date{}                                           
\begin{document}
	\thispagestyle{empty}

	\begin{abstract}
		\baselineskip=17pt	
In this paper, we establish the orbital stability of the 1-soliton solution for the derivative nonlinear Schr\"odinger equation under perturbations in $L^2(\mathbb{R})$. We demonstrate this stability by utilizing the B\"acklund transformation associated with the Lax pair and by applying the first conservation quantity in $L^2(\mathbb{R}).$
	\end{abstract}
	
	\maketitle
	
	\baselineskip=17pt

	\setcounter{tocdepth}{2} \tableofcontents

	\section {Introduction}
	\quad
In the present paper, we investigate the $L^2$-orbital stability of the 1-soliton solution for the derivative nonlinear Schrödinger (DNLS) equation
\begin{align}\label{DNLS}
	iq_t+q_{xx}+i(|q|^2q)_x=0,
\end{align}
where $q(t,x)$ is the function in dimensionless space-time variables $(t,x)$. The  equation (\ref{DNLS}) first appeared in the literature as a model for the propagation of large-wavelength Alfv\'en waves in plasma. It governs the evolution of small but finite amplitude nonlinear Alfv\'en waves that propagate quasi-parallel to the magnetic field in space plasma physics \cite{Mj1976, Mj1989}. The DNLS equation \eqref{DNLS} is a canonical dispersive equation derived from the Magneto-Hydrodynamic equations in the presence of the Hall effect \cite{Mio1976}. Additionally, it is used to describe sub-picosecond pulses in single-mode optical fibers \cite{GP, DM1983, NM1981}.

\vskip 0.1cm

\noindent {\bf Global existence and main result.}
One crucial element required for our development is a theory of existence for the initial-value problem. Our stability result is primarily related to  global existence for arbitrary data in the $L^2 $ space.

The local well-posedness of the Cauchy problem for the DNLS equation in the space  $  H^s(\mathbb{R}), s > 3/2$ was established by Tsutsumi and Fukuda using a parabolic regularization in \cite{Tsutsumi}. Later, Takaoka further generalized the local well-posedness in spaces of lower regularity, specifically $  H^s(\mathbb{R}), s \geq 1$, employing the Fourier transform restriction method \cite{Takaoka}. With the aid of Strichartz estimates, Hayashi in \cite{Hayashi} proved the global existence of solutions to the DNLS equation in $  H^1(\mathbb{R})$, provided that the initial data are small in the $L^2(\mathbb{R})$  norm. Moreover, Wu in \cite{Wu2015} removed the necessity of Hayashi's requirement and established that for $|q_0|_{L^2}\leq 4\pi$, the DNLS equation \eqref{DNLS} remains globally well-posed in $H^1(\mathbb{R})$ and extends to $H^{1/2}(\mathbb{R})$ as demonstrated in \cite{Guowu2017}.


It has been shown that the DNLS equation \eqref{DNLS} is completely integrable and admits a Lax pair \cite{KN1978}. 
Integrability is also utilized to study  global well-posedness of the DNLS equation \eqref{DNLS} \cite{Jenkins2018, LJQ,P2017, P2018}. Recently, based on its complete integrability, the well-posedness for the DNLS equation \eqref{DNLS} in low regularity cases has been intensively studied, for example, in \cite{INVENT2022,KillipIMRN,killip2023}.  The latest result in \cite{KillipDNLS} provides a unique global solution $ q \in C([0, \infty), H^s (\mathbb{R})) $ with any $ s \geq 0 $, for the Cauchy problem associated with the DNLS equation \eqref{DNLS}.
\begin{Proposition}\cite{KillipDNLS}
	The DNLS equation \eqref{DNLS} is globally well-posed in $H^s(\mathbb{R})$, for every $s\geq 0$. More precisely, given any $T>0$,  an initial data $q(0)\in L^2(\mathbb{R})$  and a sequence of Schwartz-class initial data $q_n(0)$ converging to $q(0)$ in $ L^2(\mathbb{R})$, then the  sequence of corresponding solutions $q_n(t)$ with initial data $  q_n(0) $ converges to $q(t)$ in $ L^2(\mathbb{R})$ uniformly when $0\leq t\leq T$.
\end{Proposition}

The $L^2$ space considered here seems natural to establish our stability:
the DNLS equation \eqref{DNLS} is invariant under the scaling transformation
\begin{align*}
	q(t,x)\to q_\lambda(t,x)=\sqrt{\lambda}q(\lambda^2t,\lambda x),\qquad \lambda>0,
\end{align*}
 which implies that $\| q(t,\cdot)\|_{L^2(\mathbb{R})}=  \| q_\lambda(t,\cdot)\|_{L^2(\mathbb{R})}$.
On the other hand, it is known that the infinitely many conserved quantities of  the DNLS equation \eqref{DNLS} play an important role in the well-posedness theory. The first three  are the mass,  energy and momentum conservation laws
respectively given by
\begin{align}
	&M(q):=\int_{\mathbb{R}}|q|^2\dif x,\label{def: mass}\\
	&E(q):=-\frac{1}{2}\int_{\mathbb{R}} i(\bar{q}_xq-q_x\bar{q})+|q|^4\dif x,\\
	&P(q):=\int_{\mathbb{R}} |q_x|^2+\frac{3i}{4}|q|^2(\bar{q}_xq-q_x\bar{q})+\frac{1}{2}|q|^6\dif x,
\end{align}
where the mass conservation quantity $ M $ in $ L^2 $ is only necessary for the development of stability.

The exact expression of the 1-soliton solution for the DNLS equation \eqref{DNLS} under the spectrum $z_0$ with Im$z_0^2\neq0$, is provided in \cite{K1999, P2017, H2011}.
\begin{align}
	\psi^{z_0}(t,x)=&\left( \frac{z_0e^{2(\im z_0^2 x+2\im z_0^4 t)}+\bar{z}_0e^{-2(\im z_0^2 x+2\im z_0^4 t)}}{z_0e^{-2(\im z_0^2 x+2\im z_0^4 t)}+\bar{z}_0e^{2(\im z_0^2 x+2\im z_0^4 t)}}\right)^2\times\nonumber\\ &\frac{2i(z_0^2-\bar{z}_0^2)e^{-2i(\re z_0^2 x+2\re z_0^4 t)}}{z_0e^{-2(\im z_0^2 x+2\im z_0^4 t)}+\bar{z}_0e^{2(\im z_0^2 x+2\im z_0^4 t)}}.\label{1sol}
\end{align}
\begin{remark}
Although there are different explicit forms for the 1-soliton of the DNLS equation \eqref{DNLS}, for example in \cite{INVENT2022}, it is essentially the same as (\ref{1sol}) when comparing parameters.
\end{remark}

The aim of this paper is to demonstrate the $ L^2$-orbital stability of the 1-soliton for the DNLS equation \eqref{DNLS}. The main result is stated as follows.
\begin{theorem}\label{mainthm}
	Let $z_0$ be a complex number with $\im z_0^2\neq0$, and let the associated  1-soliton $\psi^{z_0}(t, x)$ be given by (\ref{1sol}).
There   exist  a real small positive constant $\varepsilon>0$ and a constant $C$ such that for
 any   $q_0\in L^2(\mathbb{R}),$  if
	\begin{align*}
		\|q_0(\cdot) -\psi^{z_0}(0, \cdot)\|_{L^2} <  \varepsilon,
	\end{align*}
then there exists a unique global solution $q$ to the Cauchy problem associated with the DNLS equation \eqref{DNLS} with the initial value $q(0) = q_0$ and there exists a constant $z_1$ satisfying
	\begin{align*}
		\sup_{t\in\mathbb{R}^+}\inf_{a,b\in\mathbb{R}}\|q(t,\cdot)-e^{bi}\psi^{z_1}(t,\cdot+a) \|_{L^2}+|z_1-z_0| <  C \varepsilon.
	\end{align*}
\end{theorem}

\noindent {\bf Prior work and motivation.} There are two typical approaches for studying stability issues of nonlinear dispersive equations. One method is the variational approach, which constructs solitary waves as energy minimizers under appropriate constraints \cite{Lions1982, W1986}. Another approach to studying stability is to linearize the equation around the solitary waves; it is expected that nonlinear stability is governed by the linearized equation \cite{G1987}.
In recent years, many other tools have emerged for studying low-regularity orbital stability based on complete integrability. For instance, Merle and Vega in \cite{MVkdv} demonstrated $L^2$-stability and asymptotic stability of solitons for  the Korteweg-de Vries (KdV) equation by invoking the Miura transformation. This special B\"acklund transformation transforms solutions of the KdV equation into solutions of the modified KdV equation.
On another front, Hoffman and Wayne in \cite{HW2013} formulated an abstract orbital stability result for soliton solutions of integrable equations through the B\"acklund transformations. They extended their analysis to include the orbital stability of  the sine-Gordon equation and the Toda lattice. Similarly,  they in \cite{A2015} established that modified KdV breathers are $H^1$-stable, and negative energy breathers are asymptotically stable in $H^1(\mathbb{R}),$ viewing the problem from the perspective of dynamical systems. Additionally, Mizumachi and Pelinovsky in \cite{PNLS} proved $L^2$-orbital stability of the  nonlinear Schr\"odinger (NLS)  1-soliton using the B\"acklund transformation. (This transformation is also referred to as the Darboux transformation, because it can not only transform one solution to another but also transform the eigenvector to eigenvector of the corresponding Lax pair associated with these two solutions.) They further in \cite{MTM} applied this method to the Massive Thirring Model.  Koch and Tataru in \cite{TATARU} proved low-regularity orbital stability of NLS multi-solitons. Utilizing the complete integrability of the KdV equation, Killip and Vi\c{s}an in \cite{Killip2022} employed the (doubly) renormalized perturbation to establish the determinant of orbital stability of KdV multi-solitons in $H^{-1}(\mathbb{R})$. Contreras and Pelinovsky in \cite{CP2014} established $L^2$-orbital stability of multi-solitons of the NLS equation through a dressing transformation. To the best of our knowledge, these ideas have not yet been applied to the DNLS equation \eqref{DNLS}.

For the DNLS equation \eqref{DNLS}, Guo et al., via constructing three appropriate invariants of motion, obtained the stability of solitary waves in $H^1(\mathbb{R})$ for a class of solitons that admit certain limits of the spectrum parameter \cite{GuoJDE1995}. Their result is further supplemented in \cite{CO2006, KW2018}. Colin and Ohta in \cite{CO2006} relaxed the limiting conditions on the spectrum parameter in \cite{GuoJDE1995} and established the stability of solitary waves in $H^1(\mathbb{R}).$  Kwon and Wu in \cite{KW2018} studied the zero-mass case, which was not covered by the aforementioned papers. They provided orbital stability in $H^1(\mathbb{R})$  and presented a self-similar type blow-up criterion for solutions with critical mass $4\pi$. Regarding multi-solitons, with appropriate assumptions on the speeds and frequencies of the composing solitons, Le Coz and Wu established the orbital stability of multi-solitons in $H^1(\mathbb{R})$ in \cite{LW2018}.

We study the orbital stability of the 1-soliton for the DNLS equation motivated by the following problems:
\begin{itemize}
	\item  Until recently, Harrop-Griffiths et al. obtained the global well-posedness of the DNLS equation \eqref{DNLS} in $L^2$ space \cite{KillipDNLS}. Therefore, it is natural to consider the $L^2$-orbital stability of the DNLS solitons.

\item  Comparing with the classical NLS equation \cite{PNLS}, the Lax pair of the DNLS equation \eqref{DNLS} possesses more symmetries. This results in its B\"acklund transformation being more complicated than that of the NLS equation, which, in turn, causes additional difficulties when using perturbation theory to analyze the properties of the solution of the Lax pair  of the DNLS equation \eqref{DNLS}.
\end{itemize}

Our approach to the stability of solitons is inspired  by the work of Mizumachi and Pelinovsky in  \cite{PNLS}.
The main steps to deal with it  are sketched in Figure \ref{result1}.
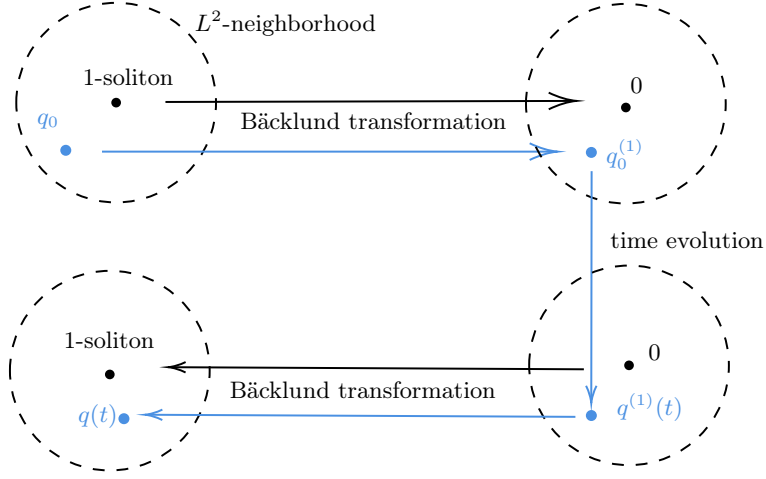
\begin{figure}
	\centering	
	\tikzset{every picture/.style={line width=0.75pt}} 
	\begin{tikzpicture}[x=0.75pt,y=0.75pt,yscale=-0.8,xscale=0.8]
		\draw  [dash pattern={on 4.5pt off 4.5pt}] (51.29,70.08) .. controls (51.29,35.31) and (79.48,7.13) .. (114.25,7.13) .. controls (149.02,7.13) and (177.21,35.31) .. (177.21,70.08) .. controls (177.21,104.85) and (149.02,133.04) .. (114.25,133.04) .. controls (79.48,133.04) and (51.29,104.85) .. (51.29,70.08) -- cycle ;
		\draw  [fill={rgb, 255:red, 0; green, 0; blue, 0 }  ,fill opacity=1 ] (111.83,70.08) .. controls (111.83,68.75) and (112.92,67.67) .. (114.25,67.67) .. controls (115.58,67.67) and (116.67,68.75) .. (116.67,70.08) .. controls (116.67,71.42) and (115.58,72.5) .. (114.25,72.5) .. controls (112.92,72.5) and (111.83,71.42) .. (111.83,70.08) -- cycle ;
		\draw  [color={rgb, 255:red, 74; green, 144; blue, 226 }  ,draw opacity=1 ][fill={rgb, 255:red, 74; green, 144; blue, 226 }  ,fill opacity=1 ] (79.67,100.17) .. controls (79.67,98.6) and (80.94,97.33) .. (82.5,97.33) .. controls (84.06,97.33) and (85.33,98.6) .. (85.33,100.17) .. controls (85.33,101.73) and (84.06,103) .. (82.5,103) .. controls (80.94,103) and (79.67,101.73) .. (79.67,100.17) -- cycle ;
		\draw    (145.33,69.67) -- (402.44,69.01) ;
		\draw [shift={(404.44,69)}, rotate = 179.85] [color={rgb, 255:red, 0; green, 0; blue, 0 }  ][line width=0.75]    (17.49,-5.26) .. controls (11.12,-2.23) and (5.29,-0.48) .. (0,0) .. controls (5.29,0.48) and (11.12,2.23) .. (17.49,5.26)   ;
		\draw [color={rgb, 255:red, 74; green, 144; blue, 226 }  ,draw opacity=1 ]   (105.33,101) -- (390.44,101) ;
		\draw [shift={(392.44,101)}, rotate = 180] [color={rgb, 255:red, 74; green, 144; blue, 226 }  ,draw opacity=1 ][line width=0.75]    (17.49,-5.26) .. controls (11.12,-2.23) and (5.29,-0.48) .. (0,0) .. controls (5.29,0.48) and (11.12,2.23) .. (17.49,5.26)   ;
		\draw  [dash pattern={on 4.5pt off 4.5pt}] (372.63,70.75) .. controls (372.63,35.98) and (400.81,7.79) .. (435.58,7.79) .. controls (470.35,7.79) and (498.54,35.98) .. (498.54,70.75) .. controls (498.54,105.52) and (470.35,133.71) .. (435.58,133.71) .. controls (400.81,133.71) and (372.63,105.52) .. (372.63,70.75) -- cycle ;
		\draw  [fill={rgb, 255:red, 0; green, 0; blue, 0 }  ,fill opacity=1 ] (433.17,73.17) .. controls (433.17,71.83) and (434.25,70.75) .. (435.58,70.75) .. controls (436.92,70.75) and (438,71.83) .. (438,73.17) .. controls (438,74.5) and (436.92,75.58) .. (435.58,75.58) .. controls (434.25,75.58) and (433.17,74.5) .. (433.17,73.17) -- cycle ;
		\draw  [color={rgb, 255:red, 74; green, 144; blue, 226 }  ,draw opacity=1 ][fill={rgb, 255:red, 74; green, 144; blue, 226 }  ,fill opacity=1 ] (411,101.5) .. controls (411,99.94) and (412.27,98.67) .. (413.83,98.67) .. controls (415.4,98.67) and (416.67,99.94) .. (416.67,101.5) .. controls (416.67,103.06) and (415.4,104.33) .. (413.83,104.33) .. controls (412.27,104.33) and (411,103.06) .. (411,101.5) -- cycle ;
		\draw [color={rgb, 255:red, 74; green, 144; blue, 226 }  ,draw opacity=1 ]   (414,113.33) -- (413.78,257.67) ;
		\draw [shift={(413.78,259.67)}, rotate = 270.09] [color={rgb, 255:red, 74; green, 144; blue, 226 }  ,draw opacity=1 ][line width=0.75]    (10.93,-3.29) .. controls (6.95,-1.4) and (3.31,-0.3) .. (0,0) .. controls (3.31,0.3) and (6.95,1.4) .. (10.93,3.29)   ;
		\draw  [dash pattern={on 4.5pt off 4.5pt}] (374.63,235.83) .. controls (374.63,201.06) and (402.81,172.88) .. (437.58,172.88) .. controls (472.35,172.88) and (500.54,201.06) .. (500.54,235.83) .. controls (500.54,270.6) and (472.35,298.79) .. (437.58,298.79) .. controls (402.81,298.79) and (374.63,270.6) .. (374.63,235.83) -- cycle ;
		\draw  [fill={rgb, 255:red, 0; green, 0; blue, 0 }  ,fill opacity=1 ] (435.17,235.83) .. controls (435.17,234.5) and (436.25,233.42) .. (437.58,233.42) .. controls (438.92,233.42) and (440,234.5) .. (440,235.83) .. controls (440,237.17) and (438.92,238.25) .. (437.58,238.25) .. controls (436.25,238.25) and (435.17,237.17) .. (435.17,235.83) -- cycle ;
		\draw  [color={rgb, 255:red, 74; green, 144; blue, 226 }  ,draw opacity=1 ][fill={rgb, 255:red, 74; green, 144; blue, 226 }  ,fill opacity=1 ] (411,267.5) .. controls (411,265.94) and (412.27,264.67) .. (413.83,264.67) .. controls (415.4,264.67) and (416.67,265.94) .. (416.67,267.5) .. controls (416.67,269.06) and (415.4,270.33) .. (413.83,270.33) .. controls (412.27,270.33) and (411,269.06) .. (411,267.5) -- cycle ;
		\draw  [dash pattern={on 4.5pt off 4.5pt}] (47.29,239.17) .. controls (47.29,204.4) and (75.48,176.21) .. (110.25,176.21) .. controls (145.02,176.21) and (173.21,204.4) .. (173.21,239.17) .. controls (173.21,273.94) and (145.02,302.13) .. (110.25,302.13) .. controls (75.48,302.13) and (47.29,273.94) .. (47.29,239.17) -- cycle ;
		\draw    (409.11,238.33) -- (149.11,237.01) ;
		\draw [shift={(147.11,237)}, rotate = 0.29] [color={rgb, 255:red, 0; green, 0; blue, 0 }  ][line width=0.75]    (10.93,-3.29) .. controls (6.95,-1.4) and (3.31,-0.3) .. (0,0) .. controls (3.31,0.3) and (6.95,1.4) .. (10.93,3.29)   ;
		\draw [color={rgb, 255:red, 74; green, 144; blue, 226 }  ,draw opacity=1 ]   (403.78,268.33) -- (134.44,267.01) ;
		\draw [shift={(132.44,267)}, rotate = 0.28] [color={rgb, 255:red, 74; green, 144; blue, 226 }  ,draw opacity=1 ][line width=0.75]    (10.93,-3.29) .. controls (6.95,-1.4) and (3.31,-0.3) .. (0,0) .. controls (3.31,0.3) and (6.95,1.4) .. (10.93,3.29)   ;
		\draw  [fill={rgb, 255:red, 0; green, 0; blue, 0 }  ,fill opacity=1 ] (107.83,241.58) .. controls (107.83,240.25) and (108.92,239.17) .. (110.25,239.17) .. controls (111.58,239.17) and (112.67,240.25) .. (112.67,241.58) .. controls (112.67,242.92) and (111.58,244) .. (110.25,244) .. controls (108.92,244) and (107.83,242.92) .. (107.83,241.58) -- cycle ;
		\draw  [color={rgb, 255:red, 74; green, 144; blue, 226 }  ,draw opacity=1 ][fill={rgb, 255:red, 74; green, 144; blue, 226 }  ,fill opacity=1 ] (116.33,269.5) .. controls (116.33,267.94) and (117.6,266.67) .. (119.17,266.67) .. controls (120.73,266.67) and (122,267.94) .. (122,269.5) .. controls (122,271.06) and (120.73,272.33) .. (119.17,272.33) .. controls (117.6,272.33) and (116.33,271.06) .. (116.33,269.5) -- cycle ;
		
		\draw (91.33,47.07) node [anchor=north west][inner sep=0.75pt]    {\footnotesize 1-soliton};
		\draw (62.67,75.4) node [anchor=north west][inner sep=0.75pt]  [color={rgb, 255:red, 74; green, 144; blue, 226 }  ,opacity=1 ]  {\footnotesize$q_{0}$};
		\draw (190.67,74.73) node [anchor=north west][inner sep=0.75pt]    {\footnotesize B\"acklund transformation};
		\draw (436.67,52.73) node [anchor=north west][inner sep=0.75pt]    {\footnotesize $0$};
		\draw (421.33,91.4) node [anchor=north west][inner sep=0.75pt]  [color={rgb, 255:red, 74; green, 144; blue, 226 }  ,opacity=1 ]  {\footnotesize $q_{0}^{( 1)}$};
		\draw (424,149.73) node [anchor=north west][inner sep=0.75pt]    {\footnotesize time evolution};
		\draw (162,10.4) node [anchor=north west][inner sep=0.75pt]    {\footnotesize $L^{2}$-neighborhood};
		\draw (428,251.4) node [anchor=north west][inner sep=0.75pt]  [color={rgb, 255:red, 74; green, 144; blue, 226 }  ,opacity=1 ]  {\footnotesize $q^{( 1)}(t)$};
		\draw (448,220.73) node [anchor=north west][inner sep=0.75pt]    {\footnotesize $0$};
		\draw (184,244.07) node [anchor=north west][inner sep=0.75pt]    {\footnotesize B\"acklund transformation};
		\draw (79.33,212.4) node [anchor=north west][inner sep=0.75pt]    {\footnotesize 1-soliton};
		\draw (88.33,258.73) node [anchor=north west][inner sep=0.75pt]  [color={rgb, 255:red, 74; green, 144; blue, 226 }  ,opacity=1 ]  {\footnotesize $q(t)$};		
	\end{tikzpicture}
	\caption{\footnotesize The B\"acklund transformation approach for the stability of the 1-soliton within a small $L^2$-neighborhood}
	\label{result1}
\end{figure}

In the first step, at time $t=0$, if the initial data $q_0$ belongs to a sufficiently small $L^2$-neighborhood of a 1-soliton, we can obtain the eigenvalue and eigenvector of its corresponding $x$-part of the Lax pair. Using this eigenvector, we construct a B\"acklund transformation that transforms the initial data $q_0$ into $q_0^{(1)}$, which resides in a small $L^2$-neighborhood of the zero solution. Furthermore, this B\"acklund transformation is Lipschitz continuous in these two $L^2$-neighborhoods.
The second step involves time evolution. The key ingredient here is the $L^2$-conservation of the DNLS equation  \eqref{DNLS}, which ensures that the solution $q^{(1)}$ of \eqref{DNLS} under the initial data $q_0^{(1)}$ remains in this small $L^2$-neighborhood of the zero solution for all times $t\geq 0$.
In the final step, we use the associated solution of the Lax pair to construct another B\"acklund transformation that transforms $q^{(1)}$ back to the solution $q$ under the initial data $q_0$. We further prove that $q$ belongs to a small $L^2$-neighborhood of a 1-soliton, and the B\"acklund transformation is Lipschitz continuous in these two $L^2$-neighborhoods.


The rest of this paper is organized as follows. In Section \ref{sec pre}, we review the process of constructing the B\"acklund transformation.
Section \ref{sec 2} explains how the B\"acklund transformation turns the area around a 1-soliton into an area around the zero solution in $L^2(\mathbb{R})$ at time $t=0$.
In Section \ref{sec 3}, assuming the initial data are smooth enough, the B\"acklund transformation creates a smooth connection between the neighborhood of the zero solution and the neighborhood of a 1-soliton, linking solutions around 1-soliton and solutions around the zero solution at all times.
The proof of Theorem \ref{mainthm} is presented in Section \ref{sec: proof}.

\vskip 0.2cm
\noindent {\bf Notations.} Throughout this paper, the following notations  will be used.
\begin{itemize}
	\item[$\bullet$]The norm in the classical space $ L^p (\mathbb{R})$, $ 1 \leqslant p \leqslant \infty $,  is denoted  as $ \|\cdot \|_{p}$. A simple  extension is that $ \|\cdot \|_{p\cap q}=\|\cdot\|_p+\|\cdot\|_q$ is the norm of  $ L^p (\mathbb{R})\cap L^q (\mathbb{R})$. Specially, the inner product in  $ L^2 (\mathbb{R})$, is given by
	\begin{align*}
	\langle	f,g\rangle=\int_{\mathbb{R}}f(s)\bar{g}(s)\dif s.
	\end{align*}
	\item[$\bullet$]The  Sobolev space $W^{k,p}(\mathbb{R})$  is a Banach space
	equipped with the norm
	\begin{equation*}
		\Vert f \Vert_{W^{k,p}}=\sum_{j=0}^{k}\Vert \partial_x^{j}f \Vert_{p}.
	\end{equation*}
Specially, $H^{k}(\mathbb{R})=W^{k,2}(\mathbb{R})$ equipped with the norm
\begin{equation*}
	\Vert f \Vert_{H^{k}}=\sum_{j=0}^{k}\Vert \partial_x^{j}f \Vert_{2}.
\end{equation*}
	\item[$\bullet$]As usual, the classical Pauli matrices $\{\sigma_j\}_{j=1,2,3}$ are defined by
	\begin{equation}\label{def:PauliM}
		\sigma_1:=\begin{pmatrix}0 & 1 \\ 1 & 0\end{pmatrix}, \quad
		\sigma_2:=\begin{pmatrix}0 & -i \\ i & 0\end{pmatrix}, \quad
		\sigma_3:=\begin{pmatrix}1 & 0 \\ 0 & -1\end{pmatrix}.
	\end{equation}
	\item[$\bullet$] The letter $C$ will be used to denote universal positive constants which may vary from line to line. We write $a\lesssim b$ to denote the inequality $a\leqslant Cb$ for some constant $C>0$. To emphasize the implied constant to depend on some parameter $\alpha$,  it is  indicated  by $C(\alpha)$.
	\item[$\bullet$] If $A$ is a matrix, then $A_{ij}$ stands for its $(i,j)$-th entry.
\end{itemize}

\section{Preliminaries}\label{sec pre}
The DNLS equitation \eqref{DNLS} is completely integrable and    admits  the Lax pair \cite{KN1978},
\begin{align}
	&\Phi_x = X(q,z) \Phi,\label{lax0x}\\
	&\Phi_t =T(q,z) \Phi, \label{lax0t}
\end{align}
where
\begin{equation}
	X(q,z)=-iz^2 \sigma_3+U(q,z),\qquad 	T(q,z)=-2iz^4\sigma_3+V(q,z),\label{def:XT}
\end{equation}
with  $z\in \mathbb{C}$ being  a spectral parameter independent of $(t,x)$ and
\begin{align}
	&U(q,z)=z\left(\begin{array}{cc}
		0 & q  \\
		-\bar{q} & 0
	\end{array}\right),\label{def:U}\\
	&V(q,z)=2z^3\left(\begin{array}{cc}
		0 & q  \\
		-\bar{q} & 0
	\end{array}\right)+iz^2|q|^2\sigma_3+iz\left(\begin{array}{cc}
	0 & q_x  \\
	\bar{q}_x & 0
\end{array}\right)-z|q|^2\left(\begin{array}{cc}
0 & q  \\
-\bar{q} & 0
\end{array}\right).\label{def V}
\end{align}
The formal compatibility condition $\Phi_{xt}=\Phi_{tx}$ yields the DNLS equitation \eqref{DNLS}.

The B\"acklund transformation for equation \eqref{DNLS} was firstly introduced in \cite{K1999}. Here, we summarize the formula of the B\"acklund transformation as follows.
For a solution $q(t,x)$ of equation \eqref{DNLS} and a spectral parameter $z$ such that $\text{Im}z^2 \neq 0$, let $\Phi := \Phi(q,z) = (\Phi_1, \Phi_2)^T$ be a $C^2$ nonzero column vector solution of the Lax pair \eqref{lax0x}-\eqref{lax0t} associated with $q$ and $z$. Then, the B\"acklund transformation
$$(q,\Phi)\to (q^{(1)},\Phi^{(1)})$$
 with  $\Phi^{(1)}:=\Phi^{(1)}(q^{(1)},z)=(\Phi^{(1)}_1,\Phi^{(1)}_2)^T$ is given by
\begin{align}
	&q^{(1)}=\left(\frac{z|\Phi_2|^2+\bar{z}|\Phi_1|^2}{z|\Phi_1|^2+\bar{z}|\Phi_2|^2} \right)^2 \left[ -q+\frac{2i(z^2-\bar{z}^2)\Phi_1\bar{\Phi}_2}{z|\Phi_2|^2+\bar{z}|\Phi_1|^2}\right] ,\label{def:darb}\\ &\Phi^{(1)}_1=\frac{\bar{\Phi}_2}{z|\Phi_1|^2+\bar{z}|\Phi_2|^2},\qquad\Phi^{(1)}_2=\frac{\bar{\Phi}_1}{\bar{z}|\Phi_1|^2+z|\Phi_2|^2}.\label{def:darb2}
\end{align}
Under this transformation, $q^{(1)}$ becomes a new solution of equation \eqref{DNLS}, and $\Phi^{(1)}$ becomes a new $C^2$ nonzero column vector solution of the Lax pair \eqref{lax0x}-\eqref{lax0t} associated with $q^{(1)}$ and $z$.

A straightforward application of the B\"acklund transformation is to construct 1-soliton and multi-soliton solutions, as demonstrated in \cite{LQP2013, K1999, H2011}. By setting $q=0$ and choosing a spectral parameter $z_0$ such that $\text{Im}z_0^2 \neq0$, we obtain a nonzero column vector solution of the Lax pair \eqref{lax0x}-\eqref{lax0t} as follows
\begin{align*}
	\Phi_1(t,x)=e^{-iz_0^2x-2iz_0^4t},\qquad \Phi_2(t,x)=e^{iz_0^2x+2iz_0^4t}.
\end{align*}
Further using  the transformation \eqref{def:darb}, the 1-soliton solution $\psi^{z_0}(t,x)$ of  \eqref{DNLS} with $ q = 0 $ and spectrum
$z_0$ is given in \eqref{1sol}.
\begin{remark}
	Note that the 1-soliton $\psi^{\bar{z}_0}(t,x)$ defined in \eqref{1sol} for the DNLS equation \eqref{DNLS} under spectrum $\bar{z}_0$ satisfies $\psi^{\bar{z}_0}(t,x)=- \psi^{z_0}(t,x)$. Therefore, without loss of generality, we assume that $\text{Im}(z_0^2) > 0$ in the proof of stability.
\end{remark}

\section{From a perturbed 1-soliton solution to a small solution at $t = 0$}\label{sec 2}
When $t=0$, we denote $\psi^{z_0}(0,x)$ as $\psi^{z_0}_0(x)$ and denote $\text{Re}(z_0^2)=\xi_0$ and $\text{Im}(z_0^2)=\eta_0$. The spectral problem \eqref{lax0x} associated with $\psi^{z_0}_0(x)$ has a fundamental solution matrix in the form
\begin{align}
	\Phi^{z_0}=(\vec{\Phi}^{z_0}_1,\vec{\Phi}^{z_0}_2), \quad \vec{\Phi}^{z_0}_j=\left(\begin{array}{cc}
		\Phi^{z_0}_{1j}   \\
		\Phi^{z_0}_{2j}
	\end{array}\right),\quad j=1,2,\label{def:Phiz0}
\end{align}
with
\begin{align*}
	&\Phi^{z_0}_{11}=\frac{e^{-i\xi_0 x-\eta_0 x}}{z_0e^{2\eta_0 x}+\bar{z}_0e^{-2\eta_0 x}},\\
	&\Phi^{z_0}_{12}=\frac{e^{-i\xi_0 x}((4z_0^2\eta_0x+\xi_0)e^{-\eta_0 x}+|z_0|^2e^{3\eta_0 x})}{z_0e^{2\eta_0 x}+\bar{z}_0e^{-2\eta_0 x}}\\
	&\Phi^{z_0}_{21}=\frac{e^{i\xi_0 x+\eta_0 x}}{\bar{z}_0e^{2\eta_0 x}+z_0e^{-2\eta_0 x}},\\
	&\Phi^{z_0}_{22}=\frac{e^{i\xi_0 x}((-4z_0^2\eta_0x+\xi_0)e^{\eta_0 x}+|z_0|^2e^{-3\eta_0 x})}{\bar{z}_0e^{2\eta_0 x}+z_0e^{-2\eta_0 x}},
\end{align*}
where $\vec{\Phi}^{z_0}_1$ is also the eigenvector of the spectral problem \eqref{lax0x} associated with the 1-soliton $\psi^{z_0}_0$. The following lemma provides an estimate on the eigenvalue and eigenvector using perturbation theory about the spectral problem \eqref{lax0x}.
\begin{lemma}\label{lemma1}
	For a given spectral parameter $z_0$ with $\text{\rm Im}(z_0^2) > 0$, there exists a positive real constant $\varepsilon$ such that for any $q_0 \in L^2(\mathbb{R})$ satisfying
\begin{align*}
|q_0 - \psi^{z_0}_0|_2 \leq \varepsilon,
\end{align*}
there exists an eigenvalue $z_1$ corresponding to an eigenvector $\tilde{\Phi} := \tilde{\Phi}(q_0, z_1) \in H^1(\mathbb{R})$ for the following spectral problem,
 $$	\tilde{\Phi}_x = X(q_0,z_1) \tilde{\Phi}$$
  with
	\begin{align}
		|z_1-z_0|+\|\tilde{\Phi}-\vec{\Phi}^{z_0}_1\|_{H^1}\lesssim \|q_0-\psi^{z_0}_0\|_2.
	\end{align}
\end{lemma}
\begin{proof}
	We rewrite the   spectral problem  \eqref{lax0x}  associated with  $q_0$ as
	\begin{align}
		\left (\sigma_3\partial_x+iz^2-z\left(\begin{array}{cc}
			0 & q_0  \\
			\bar{q}_0 & 0
		\end{array}\right) \right) \tilde{\Phi}=0\label{equ:xlax q0}.
	\end{align}
	Define  a operator relying on parameter $z$ with
		\begin{align}
		\mathcal{L}(z)=\sigma_3\partial_x+iz^2-z\left(\begin{array}{cc}
			0 & \psi^{z}_0  \\
			\bar{\psi}^{z}_0 & 0
		\end{array}\right).\label{Lz}
	\end{align}
	Then
 equation \eqref{equ:xlax q0} can be written as a perturbation equation by
	\begin{align}
		\mathcal{L}(z_0)\tilde{\Phi} =\tilde{\Delta}\tilde{\Phi},\label{equ: q0}
	\end{align}
	where
	\begin{align}
		\tilde{\Delta}=-i(z_0-z)(z_0+z)+z_0\left(\begin{array}{cc}
			0 & \psi^{z_0}_0  \\
			\bar{\psi}^{z_0}_0 & 0
		\end{array}\right)-z\left(\begin{array}{cc}
			0 & q_0  \\
			\bar{q}_0 & 0
		\end{array}\right)\label{def:Delta}.
	\end{align}
	To solve the perturbation equation \eqref{equ: q0}, we utilize the method of Lyapunov-Schmidt reductions. The initial step involves analyzing the operator $\mathcal{L}(z_0)$, which is a closed operator on $L^2(\mathbb{R})$ with domain $H^1(\mathbb{R})$.  Because $\vec{\Phi}^{z_0}_2$ grows  exponentially as $x\to\pm\infty$, it appears that Ker$\mathcal{L}(z_0)=$ span$\{\vec{\Phi}^{z_0}_1\}$. Note that the operator $\sigma_3\partial_x+iz_0^2$ is a invertible bounded operator from $H^1(\mathbb{R})$ to $L^2(\mathbb{R})$ with inverse
	\begin{align*}
		(\sigma_3\partial_x+iz_0^2)^{-1}\left(\begin{array}{cc}
			f_1 \\
			f_2
		\end{array}\right)=\left(\begin{array}{cc}
		\int^x_{+\infty}e^{-iz_0^2(x-s)}f_1(s)ds   \\
		 -\int^x_{+\infty}e^{iz_0^2(x-s)}f_2(s)ds
	\end{array}\right) .
	\end{align*}
	 A simple calculation  leads to
	 \begin{align*}
	 	z_0\left(\begin{array}{cc}
	 		0 & \psi^{z_0}_0  \\
	 		\bar{\psi}^{z_0}_0 & 0
	 	\end{array}\right)(\sigma_3\partial_x+iz_0^2)^{-1}\left(\begin{array}{cc}
	 	f_1 \\
	 	f_2
 	\end{array}\right)=z_0\left(\begin{array}{cc}
 -\psi^{z_0}_0(x)\int^x_{+\infty}e^{iz_0^2(x-s)}f_2(s)ds \\
 	\bar{\psi}^{z_0}_0(x)\int^x_{+\infty}e^{-iz_0^2(x-s)}f_1(s)ds
 \end{array}\right).
	 \end{align*}
	Then, $z_0\left(\begin{array}{cc}
		0 & \psi^{z_0}_0  \\
		\bar{\psi}^{z_0}_0 & 0
	\end{array}\right)(\sigma_3\partial_x+iz_0^2)^{-1}$  is a  Hilbert-Schmidt  operator on $L^2(\mathbb{R})$. Therefore,
\begin{align*}
	\mathcal{L}(z_0)(\sigma_3\partial_x+iz_0^2)^{-1}=I-z_0\left(\begin{array}{cc}
		0 & \psi^{z_0}_0  \\
		\bar{\psi}^{z_0}_0 & 0
	\end{array}\right)(\sigma_3\partial_x+iz_0^2)^{-1}
\end{align*}
 is a Fredholm  operator with index zero, and so does $\mathcal{L}(z_0)$. Invoking of Ker$\mathcal{L}(z_0)=$ span$\{\vec{\Phi}^{z_0}_1\}$ and  Ker$\mathcal{L}(z_0)^*=$  span$\{\sigma_1\overline{\vec{\Phi}^{z_0}_1}\}$, it follows that  Ran$\mathcal{L}(z_0)=$ span$\{\sigma_1\overline{\vec{\Phi}^{z_0}_1}\}^\perp$. Thus, we conclude that $\mathcal{L}(z_0)$ is a linear homeomorphism from
	$H^1(\mathbb{R})\cap$ span$\{\vec{\Phi}^{z_0}_1\}^\perp$ to $L^2(\mathbb{R})\cap$ span$\{\sigma_1\overline{\vec{\Phi}^{z_0}_1}\}^\perp$.
Hence one can divide $\tilde{\Phi}$ to two part in span$\{\vec{\Phi}^{z_0}_1\}^\perp$ and span$\{\vec{\Phi}^{z_0}_1\}$ respectively. Since our aim is to estimate $\tilde{\Phi}-\vec{\Phi}^{z_0}_1$, we add a normalization condition
	\begin{align}
		\langle \tilde{\Phi}-\vec{\Phi}^{z_0}_1,\vec{\Phi}^{z_0}_1\rangle=0.\label{normcondition}
	\end{align}
	We will show $	\langle \tilde{\Phi},\vec{\Phi}^{z_0}_1\rangle\neq0$ latter.
	Under this condition, we make the following decomposition:
	\begin{align*}
		\tilde{\Phi}=\Phi_s+\vec{\Phi}^{z_0}_1,
	\end{align*}
	such that $\Phi_s\perp \vec{\Phi}^{z_0}_1$. Then the perturbation equation \eqref{equ: q0} is  equivalent to
	\begin{align}
		\mathcal{L}(z_0)\Phi_s =\tilde{\Delta}(\Phi_s+\vec{\Phi}^{z_0}_1).\label{equ: q02}
	\end{align}
	Define
	\begin{align}
		F(\Phi_s,q_0-\psi^{z_0}_0,z-z_0)=\mathcal{L}(z_0)\Phi_s -\tilde{\Delta}(\Phi_s+\vec{\Phi}^{z_0}_1),
	\end{align}
	which is a $C^1$ map on $H^1(\mathbb{R})\times L^2(\mathbb{R})\times \mathbb{C}$ to $ L^2(\mathbb{R})$ near zero with
	\begin{align*}
		F(\Phi_s,0,0)=\mathcal{L}(z_0)\Phi_s,\qquad D_{\Phi_s}F(0,0,0)=\mathcal{L}(z_0).
	\end{align*}
	Where $D$ means the Frechet derivative. Define  a projection operator
	\begin{align*}
		Qf=\frac{\langle f,\sigma_1\overline{\vec{\Phi}^{z_0}_1} \rangle}{\|\sigma_1\overline{\vec{\Phi}^{z_0}_1}\|_2^2 }\sigma_1\overline{\vec{\Phi}^{z_0}_1}\ :\ L^2(\mathbb{R})\to \text{span}\{\sigma_1\overline{\vec{\Phi}^{z_0}_1}\}.
	\end{align*}
	Applying the operators $Q$
	and $I-Q$ to the original equation \eqref{equ: q02}, one obtains the equivalent system
	\begin{align}
		(I-Q)F(\Phi_s,q_0-\psi^{z_0}_0,z-z_0)=0,\qquad QF(\Phi_s,q_0-\psi^{z_0}_0,z-z_0)=0.\label{equ:phis}
	\end{align}
	For the first equation, noting  that $(I-Q)\mathcal{L}(z_0)=\mathcal{L}(z_0)$, then it is  equivalent to
	\begin{align*}
		\mathcal{L}(z_0)\Phi_s =(I-Q)\tilde{\Delta}(\Phi_s+\vec{\Phi}^{z_0}_1).
	\end{align*}
	Via the fact that Ran$(I-Q)$=Ran$\mathcal{L}(z_0)$, the above equation is  equivalent to
	\begin{align*}
		\Phi_s =\mathcal{L}(z_0)^{-1}(I-Q)\tilde{\Delta}\Phi_s+\mathcal{L}(z_0)^{-1}(I-Q)\tilde{\Delta}\vec{\Phi}^{z_0}_1.
	\end{align*}
	The operator $\mathcal{L}(z_0)^{-1}(I-Q)\tilde{\Delta}$ is bounded linear operator on $H^1(\mathbb{R})\cap$span$\{\vec{\Phi}^{z_0}_1\}^\perp$ with
	\begin{align*}
		\|\mathcal{L}(z_0)^{-1}(I-Q)\tilde{\Delta}\|\leq C(z_0)\|\tilde{\Delta}\|_{H^1(\mathbb{R})\cap\text{span}\{\vec{\Phi}^{z_0}_1\}^\perp\to L^2}\leq C(z_0) (	|z-z_0|+ \|q_0-\psi^{z_0}_0\|_2).
	\end{align*}
	Therefore, there exists a small enough $\varepsilon_0>0$ such that when $|z-z_0|+ \|q_0-\psi^{z_0}_0\|_2<\varepsilon_0$,  the operator $I-\mathcal{L}(z_0)^{-1}(I-Q)\tilde{\Delta}$ is invertible. Then  $\Phi_s$ exists uniquely with
	\begin{align*}
		\Phi_s =(I-\mathcal{L}(z_0)^{-1}(I-Q)\tilde{\Delta})^{-1}\mathcal{L}(z_0)^{-1}(I-Q)\tilde{\Delta}\vec{\Phi}^{z_0}_1,
	\end{align*}
	with
	\begin{align}
		\|\Phi_s\|_{H^1}\leq  C(z_0) (	|z-z_0|+ \|q_0-\psi^{z_0}_0\|_2).\label{est:Phis}
	\end{align}
	It also implies  the reasonability of the normalization condition in \eqref{normcondition}. Indeed, if $	\langle \tilde{\Phi},\vec{\Phi}^{z_0}_1\rangle=0$, then it is found that
	\begin{align*}
		\tilde{\Phi} =	\mathcal{L}(z_0)^{-1}(I-Q)\tilde{\Delta}\tilde{\Phi},
	\end{align*}
	namely,
	\begin{align*}
		(I-	\mathcal{L}(z_0)^{-1}(I-Q)\tilde{\Delta})\tilde{\Phi}=0.
	\end{align*}
	When $\|q_0-\psi^{z_0}_0\|_2$ is small enough, we find that $\tilde{\Phi}=0$, which gives a contradiction.
	
	The next step is to solve the bifurcation equation, namely the second equation in \eqref{equ:phis}.
	The second equation in \eqref{equ:phis} is equivalent to the orthogonality condition
	\begin{align}
		B(q_0-\psi^{z_0}_0,z-z_0):=\langle \tilde{\Delta}(\Phi_s+\vec{\Phi}^{z_0}_1),\sigma_1\overline{\vec{\Phi}^{z_0}_1}\rangle=0.\label{equ:Bif}
	\end{align}
	Note that
	\begin{align*}
		D_{z-z_0}B(0,0)=\langle 2iz_0\vec{\Phi}^{z_0}_1-\left(\begin{array}{cc}
			0 & \psi^{z_0}_0  \\
			\bar{\psi}^{z_0}_0 & 0
		\end{array}\right)\vec{\Phi}^{z_0}_1,\sigma_1\overline{\vec{\Phi}^{z_0}_1}\rangle\neq0.
	\end{align*}
	Then  by the implicit function theorem, there  exists a small positive $\varepsilon_1$ and $\delta_1<1$ such that when $\|q_0-\psi^{z_0}_0\|_2\leq \varepsilon_1$, there exists unique $z_1$ with $|z_1-z_0|\leq \delta_1$ such that  $B(q_0-\psi^{z_0}_0,z_1-z_0)=0$.
	Moreover, \eqref{equ:Bif} is equivalent to
	\begin{align*}
		&(z-z_0)	D_{z-z_0}B(0,0)+(z-z_0)^2i\langle \vec{\Phi}^{z_0}_1,\sigma_1\overline{\vec{\Phi}^{z_0}_1}\rangle\\
		=&-\langle z\left(\begin{array}{cc}
			0 & q_0-\psi^{z_0}_0  \\
			\bar{q}_0-\bar{\psi}^{z_0}_0 & 0
		\end{array}\right)(\Phi_s+\vec{\Phi}^{z_0}_1),\sigma_1\overline{\vec{\Phi}^{z_0}_1}\rangle\\
		&-(z-z_0)\langle 2iz_0\Phi_s-\left(\begin{array}{cc}
			0 & \psi^{z_0}_0  \\
			\bar{\psi}^{z_0}_0 & 0
		\end{array}\right)\Phi_s,\sigma_1\overline{\vec{\Phi}^{z_0}_1}\rangle-(z-z_0)^2i\langle \Phi_s,\sigma_1\overline{\vec{\Phi}^{z_0}_1}\rangle
	\end{align*}
	  Via substituting the estimate \eqref{est:Phis} into above equation, it is adduced that
	\begin{align*}
		|z_1-z_0|\leq C(z_0) \|q_0-\psi^{z_0}_0\|_2|z_1-z_0|+ \|q_0-\psi^{z_0}_0\|_2,
	\end{align*}
	which implies that when $\|q_0-\psi^{z_0}_0\|_2$ is sufficiently small
	\begin{align*}
		|z_1-z_0|\leq C(z_0) \|q_0-\psi^{z_0}_0\|_2.
	\end{align*}
	Combining the estimate \eqref{est:Phis} and taking $\varepsilon=\min\{\varepsilon_0,\varepsilon_1\}$,   the proof of Lemma \ref{lemma1} is complete.
\end{proof}
With  the operator   \eqref{Lz},  the  spectral problem (\ref{lax0x})  associated with $q_0$ under spectral $z_1$ can be rewrite as
\begin{align}
	\mathcal{L}(z_1)\Phi=\Delta\Phi,\qquad \Delta=z_1\left(\begin{array}{cc}
		0 & \psi^{z_1}_0-q_0  \\
		\bar{\psi}^{z_1}_0-\bar{q}_0 & 0
	\end{array}\right).\label{Phi}
\end{align}
As shown in Lemma \ref{lemma1}, it has an eigenvalue $z_1$. For convenience, denote Im$z_1^2=\eta$, Re$z_1^2=\xi$. Since $	|z_1-z_0|\lesssim\|q_0-\psi^{z_0}_0\|_2$, when $\|q_0-\psi^{z_0}_0\|_2$ is small enough, it follows that $\eta>0$.  We also denote $\Phi^{z_1}:=\Phi^{z_1}(\psi^{z_1}_0,z_1)=(\vec{\Phi}^{z_1}_1,\vec{\Phi}^{z_1}_2)$ as the fundamental solution  matrix of the equation
\begin{align}
	\mathcal{L}(z_1)\Phi^{z_1}=0\label{equ:qici},
\end{align}
 where
 \begin{align}
 	\vec{\Phi}^{z_1}_j=(\Phi^{z_1}_{1j},\Phi^{z_1}_{2j})^T, \ j=1,2,
 \end{align}
  with
\begin{align}
	\begin{aligned}
			&\Phi^{z_1}_{11}=\frac{e^{-i\xi  x-\eta  x}}{z_1e^{2\eta  x}+\bar{z}_1 e^{-2\eta  x}},\\
		&\Phi^{z_1}_{12}=\frac{e^{-i\xi  x}((4z_1^2\eta x+\xi )e^{-\eta  x}+|z_1|^2e^{3\eta  x})}{z_1e^{2\eta  x}+\bar{z}_1 e^{-2\eta  x}}\\
		&\Phi^{z_1}_{21}=\frac{e^{i\xi  x+\eta  x}}{\bar{z}_1 e^{2\eta  x}+z_1e^{-2\eta  x}},\\
		&\Phi^{z_1}_{22}=\frac{e^{i\xi  x}((-4z_1^2\eta x+\xi )e^{\eta  x}+|z_1|^2e^{-3\eta  x})}{\bar{z}_1 e^{2\eta  x}+z_1e^{-2\eta  x}}.
	\end{aligned}\label{Phiz1}
\end{align}
This enables  us to consider the eigenvector $\Phi:=\Phi(q_0,z_1)$ with a new normalization condition
\begin{align}
	\langle \Phi-\vec{\Phi}^{z_1}_1,\vec{\Phi}^{z_1}_1\rangle=0.\label{normPhi}
\end{align}
Indeed, there must be a constant $C$ such that $C\langle \Phi,\vec{\Phi}^{z_1}_1\rangle=\langle \vec{\Phi}^{z_1}_1,\vec{\Phi}^{z_1}_1\rangle$, otherwise $\langle \Phi,\vec{\Phi}^{z_1}_1\rangle=0$. Then  \eqref{Phi} also leads to that	$\Phi=	\mathcal{L}(z_1)^{-1}\Delta\Phi,$ which implies $\Phi=0$.
The following corollary is obtained by process analogous steps to Lemma \ref{lemma1}.
\begin{corollary}
	Under the conditions in Lemma \ref{lemma1}, let  $\Phi:=\Phi(q_0,z_1)$ be the eigenvector of  the spectral problem (\ref{lax0x}) associated with $q_0$  satisfying the normalization condition  \eqref{normPhi}. Then
	\begin{align*}
		\|\Phi-\vec{\Phi}^{z_1}_1\|_{H^1}\lesssim \|q_0-\psi^{z_1}_0\|_2.
	\end{align*}
\end{corollary}

Now  we shall estimate the errors between $q_0$ and the 1-soliton $\psi^{z_1}_0$.
To develop this analysis, we first prove several technical results.
 The following proposition gives the explicit column-vector
solution $w\in H^1(\mathbb{R})\cap$ span$\{\vec{\Phi}^{z_1}_1\}^\perp$ for the  inhomogeneous linear equation
\begin{align}
	\mathcal{L}(z_1)w=h\label{equ:feiqici}
\end{align}
for every column-vector function $h\in L^2(\mathbb{R})\cap$ span$\{\sigma_1\overline{\vec{\Phi}^{z_1}_1}\}^\perp$.
\begin{Proposition}\label{Prop: 2.1}
	For any $h=(h_1,h_2)^T\in L^2(\mathbb{R})\cap$ span$\{\sigma_1\overline{\vec{\Phi}^{z_1}_1}\}^\perp$,  the  inhomogeneous linear equation \eqref{equ:feiqici} has a unique solution $	w=(w_1,w_2)^T\in H^1(\mathbb{R})\cap$ span$\{\vec{\Phi}^{z_1}_1\}^\perp$ given by
	\begin{align}
		w(x)=c(h)\vec{\Phi}^{z_1}_1 +\Phi^{z_1}(x)\int^x(\Phi^{z_1}(s))^{-1}\sigma_3h(s)\dif s\label{def:int w}
	\end{align}
where $c(h)$   is a continuous linear functional on $L^2(\mathbb{R})$ with
		\begin{align*}
		c(h)=-\frac{1}{\|\vec{\Phi}^{z_1}_1\|_2^2}\langle \, \Phi^{z_1}(x)\int^x(\Phi^{z_1}(s))^{-1}\sigma_3h(s)\dif s, \; \vec{\Phi}^{z_1}_1 \,\rangle.
	\end{align*}
\end{Proposition}
\begin{proof}
	Combining the equations \eqref{equ:qici} and \eqref{equ:feiqici}, it is readily shown that
	\begin{align}
		((\Phi^{z_1})^{-1}w)_x=(\Phi^{z_1})^{-1}\sigma_3h,
	\end{align}
	which yields
	\begin{align}
		w(x)=c_1\vec{\Phi}^{z_1}_1+c_2\vec{\Phi}^{z_1}_2+\Phi^{z_1}(x)\int^x(\Phi^{z_1}(s))^{-1}\sigma_3h(s)\dif s,
	\end{align}
where the constants $c_1$ and $ c_2$ and the starting point of integration are undetermined. To ensure $w\in H^1(\mathbb{R})$, it is natural to take $c_2=0$.
	For convenience, denote
	\begin{align}
		&\Upsilon_{11}(x)=\int_{x_{11}}^x\Phi^{z_1}_{22}(s)h_1(s)\dif s,\quad \Upsilon_{12}(x)=\int_{x_{12}}^x\Phi^{z_1}_{12}(s)h_2(s)\dif s,\label{def:psi1}\\
		&\Upsilon_{21}(x)=-\int_{x_{21}}^x\Phi^{z_1}_{21}(s)h_1(s)\dif s,\quad \Upsilon_{22}(x)=-\int_{x_{22}}^x\Phi^{z_1}_{11}(s)h_2(s)\dif s.
	\end{align}
	It then follows that
	\begin{align*}
		\Phi^{z_1}(x)\int^x(\Phi^{z_1}(s))^{-1}\sigma_3h(s)\dif s=\left(\begin{array}{cc}
			\Phi^{z_1}_{11}(x)(\Upsilon_{11}(x)+\Upsilon_{12}(x)) + 	\Phi^{z_1}_{12}(x)(\Upsilon_{21}(x)+\Upsilon_{22}(x))  \\
			\Phi^{z_1}_{21}(x)(\Upsilon_{11}(x)+\Upsilon_{12}(x)) + 	\Phi^{z_1}_{22}(x)(\Upsilon_{21}(x)+\Upsilon_{22}(x))
		\end{array}\right).
	\end{align*}
	In the following, we determine the constants $x_{ij}$, $i,j=1,2$, such that each term above belongs to $L^2(\mathbb{R})$.  Since $\eta>0$, the following estimates hold
	\begin{align*}
		|\Phi^{z_1}_{11}|,\ 	|\Phi^{z_1}_{21}|\lesssim e^{-\eta |x|},\quad 	|\Phi^{z_1}_{12}|\lesssim e^{\eta x}(\chi_+(x)+x\chi_-(x)),\quad |\Phi^{z_1}_{22}|\lesssim e^{-\eta x}(x\chi_+(x)+\chi_-(x)),
	\end{align*}
	where $\chi_\pm$ are the characteristic functions on $\mathbb{R}^\pm$ respectively. It is then inferred that
	\begin{align*}
		\|\Phi^{z_1}_{11}\Upsilon_{11}\|_2 \lesssim \|e^{-\eta |x|}\int_{x_{11}}^x\chi_+(s)se^{-\eta s}h_1(s)\dif s\|_2+\|e^{-\eta |x|}\int_{x_{11}}^x\chi_-(s)e^{-\eta s}h_1(s)\dif s\|_2.
	\end{align*}
	For the first term on the right-hand side of the above inequality, it is immediately ascertained that
	\begin{align*}
		\|e^{-\eta |x|}\int_{x_{11}}^x\chi_+(s)se^{-\eta s}h_1(s)\dif s\|_2\lesssim \|e^{-\eta |x|}\|_2\|x\chi_+(x)e^{-\eta x}\|_2\|h_1\|_2.
	\end{align*}
	To ensure that  the second  term is  finite, we take $x_{11}=+\infty$. Therefore, applying Minkowski's inequality yields that
	\begin{align*}
		\|e^{-\eta |x|}\int_{x_{11}}^x\chi_-(s)e^{-\eta s}h_1(s)\dif s\|_2&=\left(\int^0_{-\infty}\big|\int^0_xe^{\eta s}h_1(s-x)\dif s \big|^2\right) ^{1/2}\\
		&\lesssim \|h_1\|_2.
	\end{align*}
	Similarly, $\|\Phi^{z_1}_{21}\Upsilon_{11}\|_2 \lesssim \|h_1\|_2$. And  taking  $x_{11}=-\infty$, it is thereby inferred that $\|\Phi^{z_1}_{11}\Upsilon_{12}\|_2,\ \|\Phi^{z_1}_{21}\Upsilon_{12}\|_2 \lesssim \|h_2\|_2.$ For $\Phi^{z_1}_{12}\Upsilon_{21}$, we find that $x_{21}$ can take $\pm\infty$ either. Without loss of generality, we take $x_{21}=+\infty$. Thus, it is found that
	\begin{align*}
		\|\Phi^{z_1}_{12}\Upsilon_{21}\|_2&\lesssim\|\chi_+(x)e^{\eta x}\int^x_{+\infty}e^{-\eta |s|}h_1(s)\dif s\|_2+\|\chi_-(x)e^{\eta x}x\int^x_{+\infty}e^{-\eta |s|}h_1(s)\dif s\|_2\\
		&\lesssim\|h_1\|_2.
	\end{align*}
	Similarly, $\|\Phi^{z_1}_{22}\Upsilon_{21}\|_2 \lesssim \|h_1\|_2$. And taking  $x_{22}=-\infty$, it is thereby inferred that $\|\Phi^{z_1}_{12}\Upsilon_{22}\|_2,\ \|\Phi^{z_1}_{22}\Upsilon_{22}\|_2 \lesssim \|h_2\|_2.$ This thus implies that $c_1$ is uniquely determined by  the orthogonality condition $\langle \vec{\Phi}^{z_1}_1, w\rangle=0$,
	namely,
	\begin{align*}
		 c_1=-\frac{1}{\|\vec{\Phi}^{z_1}_1\|_2^2}\langle\Phi^{z_1}(x)\int^x(\Phi^{z_1}(s))^{-1}\sigma_3h(s)\dif s,\vec{\Phi}^{z_1}_1\rangle,
	\end{align*}
	Since  all other terms in \eqref{def:int w} are in $L^2(\mathbb{R})$, it follows that $c_1(h)$ is a continuous linear
	functional on $L^2(\mathbb{R})$. This completes the  proof of Proposition \ref{Prop: 2.1}.
\end{proof}

We now introduce two Banach spaces of column vector functions $\mathrm{A}=\mathrm{A}_1\times\mathrm{A}_2$ and $\mathrm{B}=\mathrm{B}_1\times\mathrm{B}_2$ equipped with the norm
\begin{align}
	\|(w_1,w_2)^T\|_{\mathrm{A}}=\|w_1\|_{\mathrm{A}_1}+\|w_2\|_{\mathrm{A}_2},\qquad 	\|(h_1,h_2)^T\|_{\mathrm{B}}=\|h_1\|_{\mathrm{B}_1}+\|h_2\|_{\mathrm{B}_2},\label{def:spaceAB}
\end{align}
where  the exact expression of $\vec{\Phi}^{z_1}_1$ is given in \eqref{Phiz1} and
\begin{align*}
	&\|w_1\|_{\mathrm{A}_1}=\inf_{f_1+g_1=w_1}\left(\|(\Phi^{z_1}_{11})^{-1}f_1\|_\infty+ \|(\Phi^{z_1}_{21})^{-1}g_1\|_{\infty\cap2}\right) ,\\
	&\|w_2\|_{\mathrm{A}_2}=\inf_{f_2+g_2=w_2}\left(\|(\Phi^{z_1}_{11})^{-1}f_2\|_{\infty\cap2}+ \|(\Phi^{z_1}_{21})^{-1}g_2\|_\infty \right) ,\\
	&\|h_1\|_{\mathrm{B}_1}=\inf_{f_1+g_1=h_1}\left(\|(\Phi^{z_1}_{11})^{-1}f_1\|_{1\cap2}+ \|(\Phi^{z_1}_{21})^{-1}g_1\|_{2}\right),\\
	&\|h_2\|_{\mathrm{B}_2}=\inf_{f_2+g_2=h_2}\left(\|(\Phi^{z_1}_{11})^{-1}f_2\|_{2}+ \|(\Phi^{z_1}_{21})^{-1}g_2\|_{1\cap2}\right).
\end{align*}
By virtue of the integral expression of $w$ in Proposition \ref{Prop: 2.1}, we give the estimate of the norm of  $w$ in the following proposition.
\begin{Proposition}\label{Prop: 2.2}
	Let $h=(h_1,h_2)^T\in L^2(\mathbb{R})\cap$ span$\{\sigma_1\overline{\vec{\Phi}^{z_1}_1}\}^\perp$, and $w$ be the solution of  the  inhomogeneous linear equation \eqref{equ:feiqici} given in Proposition \ref{Prop: 2.1}. Then it appears that
	\begin{align*}
		\|w\|_{\mathrm{A}}\lesssim \|h\|_{\mathrm{B}}.
	\end{align*}
\end{Proposition}
\begin{proof}
	A direct fact is that for any $f\in L^{\infty}(\mathbb{R} )$,
	\begin{align*}
		&	\|\Phi^{z_1}_{11}f\|_{\mathrm{A}_1}=\inf_{f_1+g_1=f}\left(\|f_1\|_\infty+ \|e^{-2\eta  (\cdot)}g_1\|_{\infty\cap2}\right)\leq \|f\|_\infty,\\
		&	\|\Phi^{z_1}_{21}f\|_{\mathrm{A}_2}=\inf_{f_2+g_2=f}\left(\|e^{2\eta  (\cdot)}f_2\|_\infty+ \|g_2\|_{\infty\cap2}\right)\leq \|f\|_\infty.
	\end{align*}
	And for any $f\in L^{2}(\mathbb{R})$, since $(\Phi^{z_1}_{11})^{-1}$ and $(\Phi^{z_1}_{21})^{-1}$ have nonzero infimum, it then follows that
	\begin{align*}
		\|f\|_2\lesssim \|f\|_{\mathrm{B}_j},\ j=1,2.
	\end{align*}
	Therefore, it is accomplished that
	\begin{align}
		\|c(h)\vec{\Phi}^{z_1}_1 \|_{\mathrm{A}}\leq |c(h)|\lesssim \|h\|_2\lesssim\|h\|_{\mathrm{B}}.\label{est 1}
	\end{align}
	For the second term in \eqref{def:int w}, we take $\Phi^{z_1}_{11}(x)\Upsilon_{11}(x)$ and $	\Phi^{z_1}_{12}(x)\Upsilon_{21}(x) $ as  examples to give estimates.  The others can be obtain analogously.
	Firstly,
	\begin{align*}
		\|\Phi^{z_1}_{11}\Upsilon_{11}\|_{\mathrm{A}_1}\leq \|\Upsilon_{11}\|_\infty.
	\end{align*}
	Then for any $f_1+g_1=h_1$, from the exact expression in \eqref{def:psi1}, it is found that
	\begin{align*}
		\|\Upsilon_{11}\|_\infty = &	\big\|\int^x_{+\infty}\frac{e^{i\xi  s}((-4z_1^2\eta s+\xi )e^{\eta  x}+|z_1|^2e^{-3\eta  s})}{\bar{z}_1e^{2\eta  s}+z_1e^{-2\eta  s}}h_1(s)\dif s\|_\infty\\
		=&	\big\|\int^x_{+\infty}\frac{e^{i\xi  s}((-4z_1^2\eta s+\xi )+|z_1|^2e^{-4\eta  s})}{|\bar{z}_1e^{2\eta  s}+z_1e^{-2\eta  s}|^{2}}\left( (\bar{z}_1e^{-2\eta  s}+z_1e^{2\eta  s})e^{\eta  s}f_1(s)\right) \dif s\\
		&+\int^x_{+\infty}\frac{e^{i\xi  s}((-4z_1^2\eta s+\xi )e^{2\eta  x}+|z_1|^2e^{-2\eta  s})}{(\bar{z}_1e^{2\eta  s}+z_1e^{-2\eta  s})^{2}}\left( (\bar{z}_1e^{2\eta  s}+z_1e^{-2\eta  s})e^{-\eta  s}g_1(s)\right) \dif s\|_\infty\\
		&\lesssim \|h_1\|_{\mathrm{B}_1}.
	\end{align*}
The estimate in the last step is obtained from the fact
\begin{align*}
	&\frac{e^{i\xi  s}((-4z_1^2\eta s+\xi )+|z_1|^2e^{-4\eta  s})}{|\bar{z}_1e^{2\eta  s}+z_1e^{-2\eta  s}|^{2}}\in C_0^\infty(\mathbb{R}),\\
	&\frac{e^{i\xi  s}((-4z_1^2\eta s+\xi )e^{2\eta  x}+|z_1|^2e^{-2\eta  s})}{(\bar{z}_1e^{2\eta  s}+z_1e^{-2\eta  s})^{2}}\in\mathcal{S}(\mathbb{R}).
\end{align*}
	Next, for the term $\Phi^{z_1}_{12}\Upsilon_{21}$, it is found  that
	\begin{align*}
		\|\Phi^{z_1}_{12}\Upsilon_{21}\|_{\mathrm{A}_1}\lesssim E_1+E_2,
	\end{align*}
	where
	\begin{align*}
		E_1=\|\chi_+(x)e^{\eta  x}\int_x^{+\infty}e^{-\eta |s|}|h_1(s)|\dif s\|_{\mathrm{A}_1},\quad E_2=\chi_-(x)xe^{\eta  x}\int_x^{+\infty}e^{-\eta |s|}|h_1(s)|\dif s\|_{\mathrm{A}_1}.
	\end{align*}
	For any $f_1+g_1=|h_1|$, in view of  the fact that for $0<x<s$, $$\frac{\bar{z}_1e^{-2\eta  x}+z_1e^{2\eta  x}}{\bar{z}_1e^{-2\eta  s}+z_1e^{2\eta  s}}\lesssim e^{2\eta  (s-x)},$$ it is adduced that
	\begin{align*}
		E_1&\leq \|\chi_+(x)(\bar{z}_1e^{-2\eta  x}+z_1e^{2\eta  x})\int_x^{+\infty}e^{-\eta |s|}|h_1(s)|\dif s\|_{2\cap\infty}\\
		=&\big \|\chi_+(x)\int_x^{+\infty}\frac{\bar{z}_1e^{-2\eta  x}+z_1e^{2\eta  x}}{\bar{z}_1e^{-2\eta  s}+z_1e^{2\eta  s}}e^{-2\eta s} \left( (\bar{z}_1e^{-2\eta  s}+z_1e^{2\eta  s})e^{\eta s}f_1\right) \dif s\\
		&+\chi_+(x)\int_x^{+\infty}\frac{\bar{z}_1e^{-2\eta  x}+z_1e^{2\eta  x}}{\bar{z}_1e^{-2\eta  s}+z_1e^{2\eta  s}} \left( (\bar{z}_1e^{-2\eta  s}+z_1e^{2\eta  s})e^{-\eta s}g_1\right) \dif s\big\|_{2\cap\infty}\\
		&\lesssim 	\big \|\chi_+(x)\int_x^{+\infty}e^{-2\eta  x}\left( (\bar{z}_1e^{-2\eta  s}+z_1e^{2\eta  s})e^{\eta s}f_1\right) \dif s\big\|_{2\cap\infty}\\
		&+	\big \|\chi_+(x)\int_x^{+\infty}e^{2\eta  (s-x)} \left( (\bar{z}_1e^{-2\eta  s}+z_1e^{2\eta  s})e^{-\eta s}g_1\right) \dif s\big\|_{2\cap\infty}.
	\end{align*}
Applying  the H\"older  and Minkovski  inequalities yields the estimates in the following,
	\begin{align*}
		&\big \|\chi_+(x)\int_x^{+\infty}e^{-2\eta  x}\left( (\bar{z}_1e^{-2\eta  s}+z_1e^{2\eta  s})e^{\eta s}f_1\right) \dif s\big\|_{\infty}\\
		&\leq \big \|\chi_+(x)\int_x^{+\infty}e^{-2\eta  s}\left( (\bar{z}_1e^{-2\eta  s}+z_1e^{2\eta  s})e^{\eta s}f_1\right) \dif s\big\|_{\infty}\\
		&\lesssim\| (\bar{z}_1e^{-2\eta  s}+z_1e^{2\eta  s})e^{\eta s}f_1\|_2,\\
		&\big \|\chi_+(x)e^{-2\eta  x}\int_x^{+\infty}\left( (\bar{z}_1e^{-2\eta  s}+z_1e^{2\eta  s})e^{\eta s}f_1\right) \dif s\big\|_{2}\lesssim \| (\bar{z}_1e^{-2\eta  s}+z_1e^{2\eta  s})e^{\eta s}f_1\|_1, \\
		&\big \|\chi_+(x)\int_x^{+\infty}e^{2\eta  (s-x)} \left( (\bar{z}_1e^{-2\eta  s}+z_1e^{2\eta  s})e^{-\eta s}g_1\right) \dif s\big\|_{2}\lesssim\|  (\bar{z}_1e^{-2\eta  s}+z_1e^{2\eta  s})e^{-\eta s}g_1\|_2.
	\end{align*}
	Thus we conclude that $E_1\lesssim \|h_1\|_{\mathrm{B}_1}$. An similar estimate shows $E_2\lesssim \|h_1\|_{\mathrm{B}_1}$. Together with \eqref{est 1}, we obtain desired result as indicated above.
\end{proof}
The following proposition gives the estimates of error between $\Phi$ and $\vec{\Phi}^{z_1}_1$.
\begin{Proposition}\label{Prop 2.5}
	There exists a small $\epsilon$ such that when $$ \|q_0-\psi^{z_1}_0\|_2\leq \epsilon,$$ the column-vector function $\Phi_s=\Phi-\vec{\Phi}^{z_1}_1$  can be written in the following form
	\begin{align*}
		\Phi_s=\left(\begin{array}{cc}
			\Phi_{11}^{z_1}E_{11}-\bar{\Phi}_{21}^{z_1}E_{12} \\
			\bar{\Phi}_{11}^{z_1}E_{21}-\Phi_{21}^{z_1}E_{22}
		\end{array}\right)
	\end{align*}
satisfying
	\begin{align}
		\|E_{11}\|_\infty+\|E_{12}\|_{2\cap\infty}+	\|E_{22}\|_\infty+\|E_{21}\|_{2\cap\infty}\lesssim \|q_0-\psi^{z_1}_0\|_2.\label{est E}
	\end{align}
	In addition,  if $q_0\in H^k(\mathbb{R})$, then for $j=1,...,k$,
	\begin{align}
		\|\partial_x^jE_{11}\|_\infty+\|\partial_x^jE_{12}\|_{2\cap\infty}+	\|\partial_x^jE_{22}\|_\infty+\|\partial_x^jE_{21}\|_{2\cap\infty}<\infty.\label{est E2}
	\end{align}
\end{Proposition}
\begin{proof}
	Via the normalization condition \eqref{normPhi},  $\Phi_s$ belongs to $H^1(\mathbb{R})\cap$ span$\{\vec{\Phi}^{z_1}_1\}^\perp$ and admits
	\begin{align}
		\mathcal{L}(z_1)\Phi_s =\Delta(\Phi_s+\vec{\Phi}^{z_1}_1)=z_1\left(\begin{array}{cc}
			(\psi_0^{z_1}-q_0)(\Phi_{21}^{z_1}+\Phi_{s2} )\\
			(\bar{\psi}_0^{z_1}-\bar{q}_0)(\Phi_{11}^{z_1}+\Phi_{s1} )
		\end{array}\right).\label{equ:Phis}
	\end{align}
	As an image of $\mathcal{L}(z_1)\Phi_s$, it holds that $\mathcal{L}(z_1)\Phi_s \in  L^2(\mathbb{R})\cap$ span$\{\sigma_1\overline{\vec{\Phi}^{z_1}_1}\}^\perp$.
	By  H\"older's inequality,
	it is adduced that for any scalar function $h$,
	\begin{align*}
		&	\|(\psi_0^{z_1}-q_0)h\|_{\mathrm{B}_1}\lesssim \|\psi_0^{z_1}-q_0\|_2 \|h\|_{\mathrm{A}_2},\quad \|(\psi_0^{z_1}-q_0)h\|_{\mathrm{B}_2}\lesssim \|\psi_0^{z_1}-q_0\|_2 \|h\|_{\mathrm{A}_1}.
	\end{align*}
	Indeed, taking $\|(\psi_0^{z_1}-q_0)h\|_{\mathrm{B}_1}$ as an example, the bound follows from that
	\begin{align*}
		\|(\psi_0^{z_1}-q_0)h\|_{\mathrm{B}_1}=\inf_{f_1+g_1=h}&\left(\|(\Phi^{z_1}_{11})^{-1} (\psi_0^{z_1}-q_0)f_1\|_{1}+\|(\Phi^{z_1}_{11})^{-1} (\psi_0^{z_1}-q_0)f_1\|_{2}\right.\\
		& \left.+\|(\Phi^{z_1}_{21})^{-1}(\psi_0^{z_1}-q_0)g_1\|_{2}\right)\\
		\leq \inf_{f_1+g_1=h}&\left(\| \psi_0^{z_1}-q_0\|_{2}\|(\Phi^{z_1}_{11})^{-1}f_1\|_2+\| \psi_0^{z_1}-q_0\|_{2}\|(\Phi^{z_1}_{11})^{-1}f_1\|_{\infty}\right.\\
		& \left.+\| \psi_0^{z_1}-q_0\|_{2}\|(\Phi^{z_1}_{21})^{-1}g_1\|_{\infty}\right)\\
		&\lesssim \|\psi_0^{z_1}-q_0\|_2 \|h\|_{\mathrm{A}_1}.
	\end{align*}
	Together with the fact that $\vec{\Phi}^{z_1}_1\in \mathrm{A}$ and Proposition \ref{Prop: 2.2}, it is readily seen that
	\begin{align*}
		\|\Phi_s \|_{\mathrm{A}}\leq C\|\psi_0^{z_1}-q_0\|_2 (1+	\|\Phi_s \|_{\mathrm{A}}).
	\end{align*}

	When $\|\psi_0^{z_1}-q_0\|_2$ is small enough, it then follows that
	\begin{align*}
		\|\Phi_s \|_{\mathrm{A}}\lesssim\|\psi_0^{z_1}-q_0\|_2  .
	\end{align*}
	This completes the proof of \eqref{est E} from the definition of the norm of $\mathrm{A}$. As for \eqref{est E2}, we consider the case $k=j=1$. we differentiate the equation \eqref{equ:Phis} with respect to $x$ and obtain
	\begin{align*}
		\mathcal{L}(z_1)\partial_x\Phi_s=\partial_x(\Delta\vec{\Phi}^{z_1}_1)+\left( z_1\left(\begin{array}{cc}
			0 & \partial_x\psi^{z_1}_0  \\
			\partial_x	\bar{\psi}^{z_1}_0 & 0
		\end{array}\right)+\partial_x\Delta\right) \Phi_s+\Delta\partial_x\Phi_s.
	\end{align*}
	Analogously,
	\begin{align*}
		\big\|\left( z_1\left(\begin{array}{cc}
			0 & \partial_x\psi^{z_1}_0  \\
			\partial_x	\bar{\psi}^{z_1}_0 & 0
		\end{array}\right)+\partial_x\Delta\right) \Phi_s\big\|_{\mathrm{B}}&\lesssim(1+\|\partial_x(\psi_0^{z_1}-q_0)\|_2 )\|\Phi_s \|_{\mathrm{A}},\\
		\|\partial_x(\Delta\vec{\Phi}^{z_1}_1)\|_{\mathrm{B}}&\lesssim\|\psi_0^{z_1}-q_0\|_{H_1}.
	\end{align*}
	It  thereby transpires that
	\begin{align*}
		\|\partial_x\Phi_s \|_{\mathrm{A}}\lesssim&\|\psi_0^{z_1}-q_0\|_{H_1}+(1+\|\partial_x(\psi_0^{z_1}-q_0)\|_2 )\|\Phi_s \|_{\mathrm{A}}+\|\psi_0^{z_1}-q_0\|_2 \|\partial_x\Phi_s \|_{\mathrm{A}}.
	\end{align*}
	Then under  the smallness assumption on $\|\psi_0^{z_1}-q_0\|_2 $, it then follows that
	\begin{align*}
		\|\partial_x\Phi_s \|_{\mathrm{A}}<\infty,
	\end{align*}
	which leads to \eqref{est E2} immediately. This  completes the proof of Proposition \ref{Prop 2.5}.
\end{proof}
The next lemma is the crux of this section.
\begin{lemma}\label{lemma2.2}
	There exists a small positive constant $\epsilon$, such that if  $q_0\in L^2(\mathbb{R})$ and $\|\psi_0^{z_1}-q_0\|_2\leq \epsilon$, then $q^{(1)}_0 \in  L^2(\mathbb{R})$ satisfying
	\begin{equation}
		\|q^{(1)}_0\|_2\lesssim \|\psi_0^{z_1}-q_0\|_2, \label{norm small q}
	\end{equation}
 where $\Phi=(\Phi_1,\Phi_2)^T$ is  the solution of the  spectral problem \eqref{lax0x} under normalization condition \eqref{normPhi} and spectral $z_1$ given in Lemma \ref{lemma1}, and  $q^{(1)}_0$ is the B\"acklund transformation of $q$ under ($\Phi$, $z_1$) with
	\begin{align}
		q^{(1)}_0=\left(\frac{z_1|\Phi_2|^2+\bar{z}_1|\Phi_1|^2}{z_1|\Phi_1|^2+\bar{z}_1|\Phi_2|^2} \right)^2 \left[ \frac{2i(z_1^2-\bar{z}_1^2)\Phi_1\bar{\Phi}_2}{z_1|\Phi_2|^2+\bar{z}_1|\Phi_1|^2}-q_0\right]. \label{defq01}
\end{align}
Moreover, when $q_0\in H^3(\mathbb{R})$, $q^{(1)}_0\in H^3(\mathbb{R})$.
\end{lemma}
\begin{proof}
In view of  Proposition \ref{Prop 2.5}, a simple calculation shows that
	\begin{align}
		\frac{\Phi_1\bar{\Phi}_2}{z_1|\Phi_2|^2+\bar{z}_1|\Phi_1|^2}=\frac{\Phi_{11}^{z_1}\bar{\Phi}_{21}^{z_1}}{z_1|\Phi_{21}^{z_1}|^2+\bar{z}_1|\Phi_{11}^{z_1}|^2}+\frac{W-\frac{\Phi_{11}^{z_1}\bar{\Phi}_{21}^{z_1}}{z_1|\Phi_{21}^{z_1}|^2+\bar{z}_1|\Phi_{11}^{z_1}|^2}\tilde{W}}{1+\tilde{W}},
	\end{align}
	where
	\begin{align}
		W=\frac{1}{z_1|\Phi_{21}^{z_1}|^2+\bar{z}_1|\Phi_{11}^{z_1}|^2}&\left[(\Phi_{11}^{z_1})^2\bar{E}_{21}(1+E_{11})+ \Phi_{11}^{z_1}\bar{\Phi}_{21}^{z_1}(E_{11}-\bar{E}_{22}-E_{12}\bar{E}_{21}-E_{11}\bar{E}_{22})\right.\nonumber\\
		&\left.-(\bar{\Phi}_{21}^{z_1})^2E_{12}(1-\bar{E}_{22})\right] ,\\
		\tilde{W}=\frac{1}{z_1|\Phi_{21}^{z_1}|^2+\bar{z}_1|\Phi_{11}^{z_1}|^2}&\left[|\Phi_{11}^{z_1}|^2(2\re E_{11}+|E_{11}|^2+|E_{21}|^2)-2\re(\Phi_{11}^{z_1}\Phi_{21}^{z_1}\bar{E}_{12}(1+E_{11}))\right.\nonumber\\
		&\left.-|\Phi_{21}^{z_1}|^2(2\re E_{22}-|E_{12}|^2-|E_{22}|^2)+2\re(\Phi_{11}^{z_1}\Phi_{21}^{z_1}\bar{E}_{21}(1-E_{22})) \right] .
	\end{align}
	The  exact expressions of $\Phi_{11}^{z_1}$ and $\Phi_{21}^{z_1}$ in \eqref{Phiz1} imply that for  $j=1,2$,
	\begin{align*}
		\frac{|\Phi_{j1}^{z_1}|^2}{z_1|\Phi_{21}^{z_1}|^2+\bar{z}_1|\Phi_{11}^{z_1}|^2}\in C^{\infty}_B(\mathbb{R}),\qquad \frac{|\Phi_{11}^{z_1}\bar{\Phi}_{21}^{z_1}|}{z_1|\Phi_{21}^{z_1}|^2+\bar{z}_1|\Phi_{11}^{z_1}|^2}\in \mathcal{S}(\mathbb{R}).
	\end{align*}
	It thereby follows that
	\begin{align*}
		\|W\|_{2\cap\infty}\lesssim& \|E_{21}\|_{2\cap\infty}(1+\|E_{11}\|_\infty)+\|E_{12}\|_{2\cap\infty}(1+\|E_{22}\|_\infty)\\
		&+\|E_{11}\|_\infty+\|E_{22}\|_\infty+\|E_{11}E_{22}\|_\infty+\|E_{12}E_{21}\|_\infty,\\
		\|\tilde{W}\|_\infty\lesssim&\sum_{i,j=1,2}\|E_{ij}\|_\infty+(\sum_{i,j=1,2}\|E_{ij}\|_\infty)^2.
	\end{align*}
Then  it is readily seen from the estimates in \eqref{est E} that
\begin{align}
\label{est W}	\|W\|_{2\cap\infty},\ \|\tilde{W}\|_\infty\lesssim \|q_0-\psi^{z_1}_0\|_2.
\end{align}
	Similarly, it is found that
	\begin{align*}
		\left(\frac{z_1|\Phi_2|^2+\bar{z}_1|\Phi_1|^2}{z_1|\Phi_1|^2+\bar{z}_1|\Phi_2|^2} \right)^2=	\left(\frac{z_1|\Phi^{z_1}_{21}|^2+\bar{z}_1|\Phi^{z_1}_{11}|^2}{z_1|\Phi^{z_1}_{11}|^2+\bar{z}_1|\Phi^{z_1}_{21}|^2} \right)^2\left(\frac{1+\overline{\tilde{W}}}{1+\tilde{W}} \right) ^2,
	\end{align*}
with estimate
\begin{align*}
	\big\|\left(\frac{1+\overline{\tilde{W}}}{1+\tilde{W}} \right) ^2\big\|_\infty\lesssim1.
\end{align*}
It also follows from the exact expressions of $\Phi_{11}^{z_1}$ and $\Phi_{21}^{z_1}$ in \eqref{Phiz1} that
\begin{align*}
	\left(\frac{z_1|\Phi^{z_1}_{21}|^2+\bar{z}_1|\Phi^{z_1}_{11}|^2}{z_1|\Phi^{z_1}_{11}|^2+\bar{z}_1|\Phi^{z_1}_{21}|^2} \right)^2\in C^{\infty}_B(\mathbb{R}).
\end{align*}
	Combining with the estimates in \eqref{est E} and the fact that
	\begin{align*}
		0=\left(\frac{z_1|\Phi^{z_1}_{21}|^2+\bar{z}_1|\Phi^{z_1}_{11}|^2}{z_1|\Phi^{z_1}_{11}|^2+\bar{z}_1|\Phi^{z_1}_{21}|^2} \right)^2\left( \psi^{z_1}_0-\frac{2i(z_1^2-\bar{z}_1^2)\Phi_{11}^{z_1}\bar{\Phi}_{21}^{z_1}}{z_1|\Phi_{21}^{z_1}|^2+\bar{z}_1|\Phi_{11}^{z_1}|^2}\right) ,
	\end{align*}
	it is readily seen that
	\begin{align*}
		q^{(1)}_0=&\left(\frac{z_1|\Phi^{z_1}_{21}|^2+\bar{z}_1|\Phi^{z_1}_{11}|^2}{z_1|\Phi^{z_1}_{11}|^2+\bar{z}_1|\Phi^{z_1}_{21}|^2} \right)^2\left(\frac{1+\overline{\tilde{W}}}{1+\tilde{W}} \right) ^2\left( \psi^{z_1}_0-q_0+2i(z_1^2-\bar{z}_1^2)\frac{W-\frac{\Phi_{11}^{z_1}\bar{\Phi}_{21}^{z_1}}{z_1|\Phi_{21}^{z_1}|^2+\bar{z}_1|\Phi_{11}^{z_1}|^2}\tilde{W}}{1+\tilde{W}}\right) ,\\
		=&\left(\frac{z_1|\Phi^{z_1}_{21}|^2+\bar{z}_1|\Phi^{z_1}_{11}|^2}{z_1|\Phi^{z_1}_{11}|^2+\bar{z}_1|\Phi^{z_1}_{21}|^2} \right)^2\left(\frac{1+\overline{\tilde{W}}}{1+\tilde{W}} \right) ^2( \psi^{z_1}_0-q_0)\\
		&+\left(\frac{z_1|\Phi^{z_1}_{21}|^2+\bar{z}_1|\Phi^{z_1}_{11}|^2}{z_1|\Phi^{z_1}_{11}|^2+\bar{z}_1|\Phi^{z_1}_{21}|^2} \right)^2\left(\frac{1+\overline{\tilde{W}}}{1+\tilde{W}} \right) ^22i(z_1^2-\bar{z}_1^2)\frac{W-\frac{\Phi_{11}^{z_1}\bar{\Phi}_{21}^{z_1}}{z_1|\Phi_{21}^{z_1}|^2+\bar{z}_1|\Phi_{11}^{z_1}|^2}\tilde{W}}{1+\tilde{W}}.
	\end{align*}
Therefore, it  transpires from \eqref{est W} that
\begin{align*}
	\|	q^{(1)}_0\|_2\lesssim \|q_0-\psi^{z_1}_0\|_2+\|W\|_2+\|\tilde{W}\|_\infty\lesssim\|q_0-\psi^{z_1}_0\|_2,
\end{align*}
which gives the estimates \eqref{norm small q}. To prove $q^{(1)}_0\in H^3(\mathbb{R})$, it remains to show that $$\frac{1+\overline{\tilde{W}}}{1+\tilde{W}} \in W^{3,\infty}(\mathbb{R}),\qquad\frac{W-\frac{\Phi_{11}^{z_1}\bar{\Phi}_{21}^{z_1}}{z_1|\Phi_{21}^{z_1}|^2+\bar{z}_1|\Phi_{11}^{z_1}|^2}\tilde{W}}{1+\tilde{W}}\in H^3(\mathbb{R}).$$
	It is sufficient to prove $\partial^j_xW\in L^2(\mathbb{R})$ and $\partial^j_x\tilde{W}\in L^\infty(\mathbb{R})$ for $j=1,2,3$ which comes from the estimates  \eqref{est E2} immediately. This implies the desired result as indicated above.
\end{proof}
Furthermore, the B\"acklund transformation \eqref{def:darb2} also gives a solution $\Phi^{(1)}:=\Phi^{(1)}(q^{(1)}_0,z_1)=(\Phi^{(1)}_1,\Phi^{(1)}_2)^T$ of the spectral problem (\ref{lax0x}) associated to $q^{(1)}_0$, where
\begin{align}
	\Phi^{(1)}_1=\frac{\bar{\Phi}_2}{z_1|\Phi_1|^2+\bar{z}_1|\Phi_2|^2},\qquad\Phi^{(1)}_2=\frac{\bar{\Phi}_1}{\bar{z}_1|\Phi_1|^2+z_1|\Phi_2|^2}\label{Phi(1)}.
\end{align}
\section{From  a small solution back to perturbed 1-soliton solution}
\label{sec 3}

In this section, we apply the B\"acklund transformation provided by \eqref{def:darb}-\eqref{def:darb2} to transform a solution of the DNLS equation \eqref{DNLS} from an $L^2$-neighborhood of the zero solution to an $L^2$-neighborhood of the 1-soliton solution. To accomplish this, we first construct the Jost function of the spectral problem \eqref{lax0x} associated with a small solution $q^{(1)}_0$ at $t=0$ under the spectrum $z_1$ with Im$z_1^2>0$.
Denote its fundamental solution matrix  as $\Psi_0:=\Psi_0(q^{(1)}_0,z_1)$. Define a matrix Jost function $\mu(x)=(\mu_{ij})_{2\times2}$   by
\begin{align}
	\mu(x)=\Psi_0(x) e^{iz_1^2x\sigma_3},\label{equ Phi0}
\end{align}
which thereby satisfies
\begin{align}
	\partial_x\mu=-2iz_1^2\left(\begin{array}{cc}
		0 & \mu_{12}  \\
		-\mu_{21} & 0
	\end{array}\right)+U(q^{(1)}_0,z_1)\mu,\label{equ:mu}
\end{align}
with $U$  defined in  \eqref{def:U}. Denote $\mu=(\mu_1,\mu_2)$,
where the subscript 1 and 2 indicate the first and second columns of $\mu$, respectively.
Then  equation \eqref{equ:mu} can be converted to  two  Volterra type integral equations, respectively.
\begin{align}
	&\mu_1(x)=\left(\begin{array}{cc}
		1  \\
		0
	\end{array}\right)+\int^x_{-\infty}\left(\begin{array}{cc}
		1 &0 \\
		0&e^{2iz_1^2(x-s)}
	\end{array}\right)U(q^{(1)}_0,z_1)(s)\mu_1(s)\dif s,\label{int mu1}\\
	&\mu_2(x)=\left(\begin{array}{cc}
		0  \\
		1
	\end{array}\right)+\int^x_{+\infty}\left(\begin{array}{cc}
		e^{-2iz_1^2(x-s)} &0 \\
		0&1
	\end{array}\right)U(q^{(1)}_0,z_1)(s)\mu_2(s)\dif s.\label{int mu2}
\end{align}
The existence and uniqueness of the solutions to above equations are given in the following proposition.
\begin{Proposition}\label{Prop 3.1}
	There exists a small constant $\delta>0$ such that if $\|q^{(1)}_0\|_2<\delta$, then the integral equations \eqref{int mu1} and \eqref{int mu2} have unique solutions in the spaces $L^\infty(\mathbb{R})\times (L^\infty(\mathbb{R})\cap L^2(\mathbb{R}))$,  $(L^\infty(\mathbb{R})\cap L^2(\mathbb{R}))\times L^\infty(\mathbb{R})$, respectively satisfying
the following estimates
	\begin{align*}
		\|\mu_{11}-1\|_\infty+\|\mu_{21}\|_{2\cap\infty}+\|\mu_{22}-1\|_\infty+\|\mu_{12}\|_{2\cap\infty}\lesssim\|q^{(1)}_0\|_2.
	\end{align*}
	Furthermore, if $q^{(1)}_0\in H^3(\mathbb{R})$, then for $j=1,2,3$,
	\begin{align*}
		\|\partial_x^j\mu_{11}\|_\infty+\|\partial_x^j\mu_{21}\|_{2\cap\infty}+\|\partial_x^j\mu_{22}\|_\infty+\|\partial_x^j\mu_{12}\|_{2\cap\infty}<\infty.
	\end{align*}
\end{Proposition}
\begin{proof}
	The proof of this proposition is given by taking $\mu_1$ as an example. For convenience, denote
	\begin{align*}
		n(x)=\left(\begin{array}{cc}
			n_1(x)  \\
			n_2(x)
		\end{array}\right)=\mu_1(x)-\left(\begin{array}{cc}
			1  \\
			0
		\end{array}\right).
	\end{align*}
	Then the integral equation \eqref{int mu1} reads
	\begin{align}
		n=\mathcal{T}n+n^{(0)},\label{equ n}
	\end{align}
where $\mathcal{T}$ is a integral operator defined by
	\begin{align*}
	\mathcal{T}\left(\begin{array}{cc}
		n_1  \\
		n_2
	\end{array}\right)(x)=\int_{\mathbb{R}}\chi_+(x-s)\left(\begin{array}{cc}
			1 &0 \\
			0&e^{2iz_1^2(x-s)}
		\end{array}\right)U(q^{(1)}_0,z_1)(s)\left(\begin{array}{cc}
		n_1(s)  \\
		n_2(s)
	\end{array}\right)\dif s,
	\end{align*}
	and
	\begin{align*}
		n^{(0)}=\left(\begin{array}{cc}
			0 \\
			n^{(0)}_2(x)
		\end{array}\right)=\left(\begin{array}{cc}
			0  \\
			-z_1\int^x_{-\infty}e^{2iz_1^2(x-s)}\bar{q}^{(1)}_0(s)\dif s
		\end{array}\right).
	\end{align*}
	A directly calculation shows that
	\begin{align*}
		&	\|n^{(0)}_2\|_\infty\leq  |z_1|\|q^{(1)}_0\|_2\|\chi_-e^{2\eta x}\|_2\lesssim\|q^{(1)}_0\|_2,\\
		&\|n^{(0)}_2\|_2\leq |z_1|\int_{-\infty}^0(\int_{\mathbb{R}}|\bar{q}(s-x)|^2\dif x)^{1/2}e^{2\eta s}\dif s\lesssim\|q^{(1)}_0\|_2.
	\end{align*}
	Next, we prove that $\mathcal{T}$ is a
	contraction operator on $L^\infty(\mathbb{R})\times (L^\infty(\mathbb{R})\cap L^2(\mathbb{R}))$. For any $f=(f_1,f_2)^T\in L^\infty(\mathbb{R})\times (L^\infty(\mathbb{R})\cap L^2(\mathbb{R}))$, applying the  H\"older inequality gives that,
	\begin{align*}
		\|(\mathcal{T}f)_1\|_\infty&=\|\int^x_{-\infty}z_1q^{(1)}_0(s)f_2(s)\dif s\|_\infty\leq|z_1|\|f_2\|_2\|q^{(1)}_0\|_2,\\
		\|(\mathcal{T}f)_2\|_\infty&=\|\int^x_{-\infty}e^{2iz_1^2(x-s)}z_1\bar{q}^{(1)}_0(s)f_1(s)\dif s \|_\infty\\
		&\leq|z_1| \|f_1\|_\infty\int^x_{-\infty}e^{-2\eta (x-s)}|q^{(1)}_0(s)|\dif s\lesssim\|f_1\|_\infty\|q^{(1)}_0\|_2,\\
		\|(\mathcal{T}f)_2\|_2&\leq \|\int^0_{-\infty}e^{-2iz_1^2s}z_1\bar{q}^{(1)}_0(s-x)f_1(s-x)\dif s\|_2\\
		&\lesssim\|f_1\|_\infty\|q^{(1)}_0\|_2.
	\end{align*}
	Therefore, when $\|q^{(1)}_0\|_2$ is sufficiently
	small, $\mathcal{T}$ is a
	contraction operator, which implies that $I-\mathcal{T}$ is reversible. It thereby appears that
	\begin{align*}
		\|n_1\|_\infty+\|n_2\|_{2\cap\infty}\lesssim\|n_2^{(0)}\|_{2\cap\infty}\lesssim\|q^{(1)}_0\|_2.
	\end{align*}
		When $q^{(1)}_0\in H^3(\mathbb{R})$, taking the derivative of  \eqref{equ n} and integral by parts yields
	\begin{align*}
		\partial_x n&=\partial_xn^{(0)}+\left(\begin{array}{cc}
			z_1q^{(1)}_0(x)n_2(x)  \\
			-	\int^x_{-\infty}2iz_1^3e^{2iz_1^2(x-s)}\bar{q}^{(1)}_0(s)n_1(s)\dif s-z_1\bar{q}^{(1)}_0(x)n_1(x)
		\end{array}\right)\\
	&=(\partial_x\mathcal{T})n+\partial_xn^{(0)}+\mathcal{T}(\partial_xn),
	\end{align*}
	where
	\begin{align*}
		&(\partial_x\mathcal{T})n=\left(\begin{array}{cc}
		z_1\int^x_{-\infty}\partial_s(q^{(1)}_0)(s)n_2(s)  \\
			-	\int^x_{-\infty}z_1e^{2iz_1^2(x-s)}\partial_s(\bar{q}^{(1)}_0)(s)n_1(s)\dif s
		\end{array}\right),\\
		&\partial_x	n^{(0)}=\left(\begin{array}{cc}
			0 \\
			-2iz_1^3\int^x_{-\infty}e^{2iz_1^2(x-s)}\bar{q}^{(1)}_0(s)\dif s -z_1\bar{q}^{(1)}_0(x)
		\end{array}\right),
	\end{align*}
	are both in $L^\infty(\mathbb{R})\times (L^\infty(\mathbb{R})\cap L^2(\mathbb{R}))$ via analogous calculations.
	This in turn implies that $ \partial_x n\in L^\infty(\mathbb{R})\times (L^\infty(\mathbb{R})\cap L^2(\mathbb{R}))$. Similarly, it follows that  $ \partial_x^j n\in L^\infty(\mathbb{R})\times (L^\infty(\mathbb{R})\cap L^2(\mathbb{R}))$, $j=2,3$.
	This completes the proof of Proposition \ref{Prop 3.1}.
\end{proof}
Our attention is now turned to the time evolution.
Let $q^{(1)}_0\in H^3(\mathbb{R})$ be the  initial
data of the  DNLS equation \eqref{DNLS}. Then  it is shown in \cite{KillipDNLS, O1996}  that there  exists a  unique solution $q^{(1)}\in C([0,\infty), H^3(\mathbb{R}))$ of equation \eqref{DNLS} with $\|q^{(1)}\|_2=\|q^{(1)}_0\|_2$, which implies that $q^{(1)}\in C^1([0,\infty), H^1(\mathbb{R}))$. Then we could begin to consider the solution of the Lax pair \eqref{lax0x}-\eqref{lax0t} for every $t\geq 0$.
We use $\Psi_0(x)$ obtained in \eqref{equ Phi0} from $\mu(x)$ to be the  initial
data to solve the $t$-part of the Lax pair  in \eqref{lax0t} and obtain the Cauchy problem:
\begin{align}
	 \left\{ \begin{array}{ll}
	\partial_t\Psi(t,x)=T(q^{(1)},z_1)(t,x)\Psi(t,x)	,\\
	\Psi(0,x)=\Psi_0(x).\\
	\end{array}\right.
\end{align}
We continue to define
\begin{align}
	\mu(t,x)=\Psi(t,x) e^{iz_1^2x\sigma_3}.\label{equ Phit}
\end{align}
It also satisfies the same form as the $t$-part of the Lax pair  in \eqref{lax0t}:
\begin{align}
	\mu_t=-2iz_1^4\sigma_3\mu+V(q^{(1)},z_1)\mu,\label{mut}
\end{align}
which leads to the following integral equation
\begin{align}
	\mu(t,x)=e^{-2iz_1^4t\sigma_3}\mu(0,x)+\int_0^te^{-2iz_1^4\sigma_3(\tau-t)}V(q^{(1)},z_1)(\tau,x)\mu(\tau,x)\dif \tau.\label{int mu t}
\end{align}
The following lemma characterizes the matrix function  $\mu(t,x)$.
\begin{lemma}\label{lemma2}
	If  $q^{(1)}\in  C([0,\infty),H^3(\mathbb{R}))$ with $\|q^{(1)}\|_2=\|q^{(1)}_0\|_2$ sufficient
	small, the matrix function $\mu(t,x)$ given by the integral equation \eqref{int mu t} exists uniquely such that for every $t\geq0$, $\mu_{11}(t,\cdot),\ \mu_{22}(t,\cdot)\in L^\infty(\mathbb{R})$ and  $\mu_{21}(t,\cdot),\ \mu_{12}(t,\cdot)\in L^\infty(\mathbb{R})\cap L^2(\mathbb{R})$. For every $t\geq0$, $\mu(t,x)$ as a function of $x$ satisfies the $x$-equation in \eqref{equ:mu} subject to the boundary values
	\begin{align}
		\lim_{x\to-\infty}\mu_1(t,x)=\left(\begin{array}{cc}
			e^{-2iz_1^4t}  \\
			0
		\end{array}\right),\qquad	\lim_{x\to+\infty}\mu_2(t,x)=\left(\begin{array}{cc}
			0\\
			e^{2iz_1^4t}
		\end{array}\right),\label{limx mu}
	\end{align}
	with
	\begin{align}	\|\mu_{11}(t,\cdot)-e^{-2iz_1^4t}\|_\infty+\|\mu_{21}(t,\cdot)\|_{2\cap\infty}+\|\mu_{22}(t,\cdot)-e^{2iz_1^4t}\|_\infty+\|\mu_{12}(t,\cdot)\|_{2\cap\infty}\lesssim\|q^{(1)}\|_2.\label{estnorm}
	\end{align}
\end{lemma}
\begin{proof}
	Similar to Proposition \eqref{Prop 3.1}, we take $\mu_1(t,x)$ as an example to give the proof.
	For convenience, we denote  $V(q^{(1)},z_1)=(V_{ij})_{2\times 2}$, which is in $C^1(\mathbb{R}^+,L^2(\mathbb{R}))\cap C(\mathbb{R}^+,H^2(\mathbb{R}))$. The existence of $\mu_1(t,x)$ can be constructed  by iterative sequence.
	 Firstly, we will show that $\mu_1(t,\cdot)\in L^\infty(\mathbb{R})$. Since $\mu_1(0,x)$ is bounded,  in view of  Proposition \ref{Prop 3.1}, it is also a continuous function. This thus implies  from \eqref{mut} that
	\begin{align*}
		\partial_t(|\mu_{11}|^2+|\mu_{21}|^2)=&4\im z_1^4(|\mu_{11}|^2-|\mu_{21}|^2)+2\re V_{11}(|\mu_{11}|^2-|\mu_{21}|^2)\\
		&+2\re\mu_{11}\bar{\mu}_{21}(\bar{V}_{12}+V_{21}).
	\end{align*}
	Invoking the exact expression of $V(q^{(1)},z_1)$ in \eqref{def V}, it is adduced that
	\begin{align*}
		\partial_t(|\mu_{11}|^2+|\mu_{21}|^2)\leq\left( 4 |\im z_1^4|+|q^{(1)}|^2|\im z_1^2|+4|\im z_1^3||q^{(1)}|+2|\im z_1|(|q^{(1)}_x|+|q^{(1)}|^3) \right)  (|\mu_{11}|^2+|\mu_{21}|^2).
	\end{align*}
	Applying  Gronwall's inequality, for any $T_0 >0$, reveals  that for every $t\in[0,T_0]$,
	\begin{align*}
		|\mu_{11}(t,x)|^2+|\mu_{21}(t,x)|^2\leq e^{C(T_0)T_0}(|\mu_{11}(0,x)|^2+|\mu_{21}(0,x)|^2),
	\end{align*}
	where
	\begin{align*}
		C(T_0)=\sup_{t\in[0,T_0],x\in\mathbb{R}}\left( 4 |\im z_1^4|+|q^{(1)}(t,x)|^2|\im z_1^2|+4|\im z_1^3||q^{(1)}(t,x)|+2|\im z_1|(|q^{(1)}_x(t,x)|+|q^{(1)}(t,x)|^3) \right)
	\end{align*}
	remains bounded for any finite $T_0$. 
Thus for any $t>0$,  $\mu_1(t,\cdot)\in L^\infty(\mathbb{R})$.

Secondly, we will prove that $\mu_1(t,x)$ as a function of $x$ satisfies the $x$-equation in \eqref{equ:mu}, namely,
	\begin{align}
		\partial_x\mu_1=-2iz_1^2\left(\begin{array}{cc}
			0   \\
			-\mu_{21}
		\end{array}\right)+U(q^{(1)},z_1)\mu_1.
	\end{align}
To this end, it follows from $\mu_1(t,\cdot)\in L^\infty(\mathbb{R})$ and $\mu_1$ admitting the differential equation \eqref{mut}  that $\mu_1(\cdot,x)\in C^1(\mathbb{R}^+)$. Furthermore, via $V(q^{(1)},z_1)(t\cdot),\ \mu(0,\cdot)\in C^1(\mathbb{R})$, it is inferred  that $\mu_1(t,\cdot)\in C^1(\mathbb{R})$, and $\mu_1\in C^1(\mathbb{R}^+\times\mathbb{R})$. Define a residual function
	\begin{align*}
		R=\left(\begin{array}{cc}
			R_1  \\
			R_2
		\end{array}\right)=\partial_x\mu_1+2iz_1^2\left(\begin{array}{cc}
			0   \\
			-\mu_{21}
		\end{array}\right)-U(q^{(1)},z_1)\mu_1.
	\end{align*}
	Then $R(0,x)\equiv0$. Recalling  the matrix function $X$ and $T$ defined in \eqref{def:XT}, $R(t,x)$ can be rewriten as
	\begin{align*}
		R=\partial_x\mu_1-X\mu_1-iz_1^2\mu_1.
	\end{align*}
It is noted that $q^{(1)}\in  C([0,\infty),H^3(\mathbb{R}))$ being the  solution of the DNLS equation \eqref{DNLS} is equivalent to
	$\partial_xT-\partial_tX+[T,X]=0$. Therefore, it is adduced that
	\begin{align*}
		\partial_tR&=\partial_x(T)\mu_1+T\partial_x(\mu_1)-\partial_t(X)\mu_1-XT\mu_1-iz_1^2T\mu_1\\
		&=T(\partial_x\mu_1-X\mu_1-iz_1^2\mu_1)=T(\partial_x\mu_1-X\mu_1-iz_1^2\mu_1)\\
		&=TR.
	\end{align*}
	An analogous calculation then shows that
	\begin{align*}
		\partial_t(|R_1|^2+|R_2|^2)\leq \left( 4 |\im z_1^4|+|q^{(1)}|^2|\im z_1^2|+4|\im z_1^3||q^{(1)}|+2|\im z_1|(|q^{(1)}_x|+|q^{(1)}|^3) \right) (|R_1|^2+|R_2|^2).
	\end{align*}
	Again by Gronwall's inequality, for any $T_0 >0$, it follows that for every $t\in[0,T_0]$,
	\begin{align*}
		\partial_t(|R_1(t,x)|^2+|R_2(t,x)|^2)\leq e^{C(T_0)T_0}(|R_1(0,x)|^2+|R_2(0,x)|^2)=0.
	\end{align*}
	Then $R(t,x)\equiv0$, which means $\mu(t,x)$  admits the $x$-equation in \eqref{equ:mu}. Thirdly, we will verify that $\mu_{21}(t,\cdot)\in L^2(\mathbb{R})$. To see this, we use  the following estimate
	\begin{align*}
		\partial_t|\mu_{21}|^2\leq &(|\im z_1^4|+2|q^{(1)}|^2|\im z_1^2|+|\im z_1^3||q^{(1)}|+|\im z_1|(|q^{(1)}_x|+|q^{(1)}|^3)) |\mu_{21}|^2\\
		& +(|\im z_1^3||q^{(1)}|+|\im z_1|(|q^{(1)}_x|+|q^{(1)}|^3))|\mu_{11}|^2.
	\end{align*}
An application of Gronwall's inequality, for any $T_0 >0$ then reveals  that for every $t\in[0,T_0]$,
	\begin{align*}
		|\mu_{21}(t,x)|^2\leq& e^{C(T_0)T_0}|\mu_{21}(0,x)|^2 \\
		&+e^{C(T_0)T_0}\int_0^{T_0}(|\im z_1^3||q^{(1)}(\tau,x)|+|\im z_1|(|q^{(1)}_x(\tau,x)|+|q^{(1)}(\tau,x)|^3))|\mu_{11}(\tau,x)|^2\dif \tau .
	\end{align*}
	Then for any finite time $T_0 > 0$, we have
	\begin{align*}
		\|\mu_{21}(t,\cdot)\|_2&\leq e^{C(T_0)T_0}\|\mu_{21}(0,\cdot)\|_2 \\
		&+e^{C(T_0)T_0}\int_0^{T_0}(|\im z_1^3|\|q^{(1)}(\tau,\cdot)\|_2+|\im z_1|(\|q^{(1)}_x(\tau,\cdot)\|_2+\|q^{(1)}(\tau,\cdot)\|_2^3))\|\mu_{11}(\tau,\cdot)\|_{\infty}\dif \tau. \\
		&<\infty
	\end{align*}
which proves that $\mu_{21}(t,\cdot)\in L^2(\mathbb{R})$. Finally,  via Lebesgue's dominated convergence theorem, the integral equation \eqref{int mu t} and the fact that for any $T_0>0$, $V(q^{(1)},z_1)\in L^\infty([0,T_0]\times\mathbb{R})$ and $\underset{x\to\pm\infty}{\lim}V(q^{(1)},z_1)(t,x)=0$, $t\in[0,T_0]$, we obtain
	\begin{align*}
		\lim_{x\to-\infty}(\mu_1(t,x)-e^{-2iz_1^4t\sigma_3}\mu_1(0,x))=0.
	\end{align*}
	
Since from \eqref{int mu1}, $\underset{x\to-\infty}{\lim}\mu_1(0,x)=(1,0)^T$, it then gives   \eqref{limx mu}. Hence
	$\mu_1$ satisfies
	\begin{align}
		&\mu_1(t,x)=\left(\begin{array}{cc}
			e^{-2iz_1^4t}  \\
			0
		\end{array}\right)+\int^x_{-\infty}\left(\begin{array}{cc}
			1 &0 \\
			0&e^{2iz_1^2(x-s)}
		\end{array}\right)U(q^{(1)},z_1)(s)\mu_1(t,s)\dif s.
	\end{align}
	Thus, \eqref{estnorm} can be proved in exactly the same way as Proposition \ref{Prop 3.1}.  Consequently, this completes the proof of Lemma \ref{lemma2}.
\end{proof}
Using the Jost function $\mu(t,x)$, we can construct the B\"acklund transformation of $q^{(1)}(t,x)$. Consider the column vector-valued function $\nu=(\nu_1,\nu_2)^T$, with
\begin{align}
	\nu= e^{a_1+b_1i}e^{-iz_1^2x}\mu_{1}+e^{a_2+b_2i}e^{iz_1^2x}\mu_{2},\label{def: phi}
\end{align}
where $a_1$ and $a_2$ are two nonzero real constants and $b_1,\ b_2\in\mathbb{R}$. Then from Lemma \ref{lemma2},
$\nu$ is also a solution of the Lax pair \eqref{lax0x}-\eqref{lax0t} associated to $q^{(1)}(t,x)$ and spectrum $z_1$.
As shown in \eqref{def:darb}-\eqref{def:darb2}, the B\"acklund transformation of $(q^{(1)}(t,x),\nu)$ is given by
\begin{align}
	&Q=\left(\frac{z_1|\nu_2|^2+\bar{z}_1|\nu_1|^2}{z_1|\nu_1|^2+\bar{z}_1|\nu_2|^2} \right)^2  \left[ \frac{2i(z_1^2-\bar{z}_1^2)\nu_1\bar{\nu}_2}{z_1|\nu_2|^2+\bar{z}_1|\nu_1|^2}-q^{(1)}\right] ,\label{trans q1}\\
	&\Psi^{(Q)}_1=\frac{\bar{\nu}_2}{z_1|\nu_1|^2+\bar{z}_1|\nu_2|^2},\qquad\Psi^{(Q)}_2=\frac{\bar{\nu}_1}{\bar{z}_1|\nu_1|^2+z_1|\nu_2|^2}.\label{trans Qphi}
\end{align}
Here $\Psi^{(Q)}:=\Psi^{(Q)}(Q,z_1)=(\Psi^{(Q)}_1,\Psi^{(Q)}_2)^T$ is the solution of the Lax pair \eqref{lax0x}-\eqref{lax0t} associated to $Q$.

It  then follows from Proposition \ref{Prop 3.1} and Lemma \ref{lemma2} that $\nu\neq 0$, $\partial_x^j\nu_1\in L^\infty(\mathbb{R})\times (L^\infty(\mathbb{R})\cap L^2(\mathbb{R}))$ while $\partial_x^j\nu_2\in (L^\infty(\mathbb{R})\cap L^2(\mathbb{R}))\times L^\infty(\mathbb{R})$, $j=1,2,3$. Then $\Psi^{(Q)}$ has the same properties as $\nu$. This in turn  implies $\partial_t\partial_x\Psi^{(Q)}=\partial_x\partial_t\Psi^{(Q)}$, whose compatibility condition yields that $Q$ is a solution of the  DNLS equation \eqref{DNLS}. The following lemma shows that  $Q$ belongs to a $L^2$-neighborhood of a 1-soliton.
\begin{lemma}\label{lemma3.3}
	Suppose that  $q^{(1)}\in  C([0,\infty),H^3(\mathbb{R}))$ is a  solution of the  DNLS equation \eqref{DNLS} with $\|q^{(1)}\|_2<\delta$ for a  sufficiently small $\delta>0$. Then $Q$ defined in \eqref{trans q1} belongs to $ H^3(\mathbb{R})$ with
	\begin{align}
		\big\|Q-e^{(b_1-b_2+\frac{\re z_1^2(a_1-a_2)}{2\im z_1^2})i}\psi^{z_1}(t,x+\frac{a_1-a_2}{2\im z_1^2})\|_2\leq C(\delta)\|q^{(1)}\big\|_2.
	\end{align}
\end{lemma}
\begin{proof}
	$Q \in H^3(\mathbb{R})$ could be verified  from its definition and the property of $\Psi^{(Q)}$ immediately. On the other hand,  denoting
	\begin{align*}
		\theta=i(z_1^2x+2z_1^4t),
	\end{align*}
	a direct calculation then gives  that
		\begin{align}
	&e^{(b_1-b_2+\frac{\re z_1^2(a_1-a_2)}{2\im z_1^2})i}\psi^{z_1}(t,x+\frac{a_1-a_2}{2\im z_1^2})=\nonumber\\
	&\left( \frac{z_1e^{a_2-a_1}e^{2\re\theta}+\bar{z}_1e^{a_1-a_2}e^{-2\re\theta}}{z_1e^{a_2-a_1}e^{-2\re\theta}+\bar{z}_1e^{a_1-a_2}e^{2\re\theta}}\right) ^2\times
	\frac{2i(z_1^2-\bar{z}_1^2)e^{(b_1-b_2)i}e^{-2\im\theta i}}{z_1e^{a_2-a_1}e^{2\re\theta}+\bar{z}_1e^{a_1-a_2}e^{-2\re\theta}}.\label{solit}
\end{align}
We now rewrite the expression
\begin{align}
		\frac{\nu_1\bar{\nu}_2}{z_1|\nu_2|^2+\bar{z}_1|\nu_1|^2}=&\frac{e^{(b_1-b_2)i}e^{-2\im\theta i}}{z_1e^{a_2-a_1}e^{2\re\theta}+\bar{z}_1e^{a_1-a_2}e^{-2\re\theta}}\nonumber\\
		&+\frac{W_1-W_2\frac{e^{(b_1-b_2)i}e^{-2\im\theta i}}{z_1e^{a_2-a_1}e^{2\re\theta}+\bar{z}_1e^{a_1-a_2}e^{-2\re\theta}}}{1+W_2},\label{equ 1}\\
		\left(\frac{z_1|\nu_2|^2+\bar{z}_1|\nu_1|^2}{z_1|\nu_1|^2+\bar{z}_1|\nu_2|^2} \right)^2 =&\left( \frac{z_1e^{a_2-a_1}e^{2\re\theta}+\bar{z}_1e^{a_1-a_2}e^{-2\re\theta}}{z_1e^{a_2-a_1}e^{-2\re\theta}+\bar{z}_1e^{a_1-a_2}e^{2\re\theta}}\right) ^2\left( \dfrac{1+\bar{W}_2}{1+W_2}\right) ^2,\label{equ 2}
	\end{align}
where
	\begin{align*}
	W_1=\frac{1}{z_1e^{2a_2}e^{2\re\theta}+\bar{z}_1e^{2a_1}e^{-2\re\theta}}&\left[\left(  c_1\bar{c}_2(\mu_{11}-e^{-2iz_1^4t})e^{-2\xi  ix}+|c_2|^2\mu_{12}e^{-2\im z_1^2x}e^{-ix\xi }\right) e^{2i\bar{z}_1^4t}\right.\\
	&+e^{-2iz_1^4t}\left(|c_1|^2 \bar{\mu}_{21}e^{2\im z_1^2x}+ c_1\bar{c}_2\overline{(\mu_{22}-e^{2iz_1^4t})}e^{2i\xi x}\right) \\
	&+|c_1|^2(\mu_{11}-e^{-2iz_1^4t})\bar{\mu}_{21}e^{2\im z_1^2x}+ c_1\bar{c}_2\overline{(\mu_{22}-e^{2iz_1^4t})}(\mu_{11}-e^{-2iz_1^4t})e^{-ix\xi }\\
	&\left.+c_2\bar{c}_1\mu_{12}\bar{\mu}_{21}e^{2\im z_1^2x}+|c_2|^2\mu_{12}\overline{(\mu_{22}-e^{2iz_1^4t})}e^{-2\im z_1^2x}\right] ,\\
	W_2=\frac{\bar{z}_1}{z_1e^{2a_2}e^{2\re\theta}+\bar{z}_1e^{2a_1}e^{-2\re\theta}}&\left[2\re\left(  c_1e^{-2iz_1^4t}\left(\overline{c_1(\mu_{11}-e^{-2iz_1^4t})}e^{2\im z_1^2x}+ \bar{c}_2\bar{\mu}_{12}e^{-2ix\xi }\right)\right) \right.\\
	&+|c_1(\mu_{11}-e^{-2iz_1^4t})|^2e^{2\im z_1^2x}+|c_2\mu_{12}|^2e^{-2\im z_1^2x}\\
	&\left.+2\re (c_1c_2(\mu_{11}-e^{-2iz_1^4t})\mu_{12}) \right] \\
	+\frac{z_1}{z_1e^{2a_2}e^{2\re\theta}+\bar{z}_1e^{2a_1}e^{-2\re\theta}}&\left[ 2\re c_2e^{2iz_1^4t}\left(\overline{c_1\mu_{21}}e^{2i\xi  x} +\bar{c}_2\overline{(\mu_{22}-e^{2iz_1^4t})}e^{-2\im z_1^2 x}\right) \right.\\
	&+|c_1\mu_{21}|^2e^{2\im z_1^2 x}+|c_2(\mu_{22}-e^{2iz_1^4t})|e^{-2\im z_1^2x}\\
	&\left.+2\re(c_1c_2\mu_{21}(\mu_{22}-e^{2iz_1^4t}))\right] .
\end{align*}
Here for convenience we denote $c_j=e^{a_j+b_ji}$, $j=1,2$.

We now know that as  functions of $x$,
	\begin{align*}
		\frac{1}{z_1e^{2a_2}e^{2\re\theta}+\bar{z}_1e^{2a_1}e^{-2\re\theta}}\in \mathcal{S}(\mathbb{R}),\qquad	\frac{e^{\pm2\im z_1^2x}}{z_1e^{2a_2}e^{2\re\theta}+\bar{z}_1e^{2a_1}e^{-2\re\theta}}\in C_0^\infty(\mathbb{R}),
	\end{align*}
	and as  functions of $t$, on the other hand,
	\begin{align*}
		\frac{1}{z_1e^{2a_2}e^{2\re\theta}+\bar{z}_1e^{2a_1}e^{-2\re\theta}},\qquad	\frac{e^{\pm2 \im z_1^4t}}{z_1e^{2a_2}e^{2\re\theta}+\bar{z}_1e^{2a_1}e^{-2\re\theta}}\in \mathcal{S}(\mathbb{R}).
	\end{align*}
	It then  follows that
	\begin{align*}
		\|W_1\|_2\lesssim&\|\mu_{11}-e^{-2iz_1^4t}\|_\infty+\|\mu_{22}-e^{2iz_1^4t}\|_\infty+\|\mu_{21}\|_2+\|\mu_{12}\|_2\\
		&+\|\mu_{11}-e^{-2iz_1^4t}\|_\infty(\|\mu_{21}\|_2+\|\mu_{22}-e^{2iz_1^4t}\|_\infty)+\|\mu_{12}\|_2(\|\mu_{21}\|_\infty+\|\mu_{22}-e^{2iz_1^4t}\|_\infty),\\
		\|W_2\|_\infty\lesssim&\|\mu_{11}-e^{-2iz_1^4t}\|_\infty+\|\mu_{22}-e^{2iz_1^4t}\|_\infty+\|\mu_{21}\|_\infty+\|\mu_{12}\|_\infty\\
		&+(\|\mu_{11}-e^{-2iz_1^4t}\|_\infty+\|\mu_{22}-e^{2iz_1^4t}\|_\infty+\|\mu_{21}\|_\infty+\|\mu_{12}\|_\infty)^2.
	\end{align*}
In view of  \eqref{estnorm}, we conclude that when $\|q^{(1)}\|_2$ is small,
	\begin{align*}
		\|W_1\|_2,\ \|W_2\|_\infty\lesssim&\|q^{(1)}\|_2.
	\end{align*}
This in turn  implies that
\begin{align*}
\big\|\left( \dfrac{1+\bar{W}_2}{1+W_2}\right) ^2-1\big\|_\infty=\big\|\left(\dfrac{1+\bar{W}_2}{1+W_2}+1 \right) \dfrac{\bar{W}_2-W_2}{1+W_2} \big\|_\infty\lesssim\|q^{(1)}\|_2.
\end{align*}
Substituting \eqref{solit} and  equality \eqref{equ 1}-\eqref{equ 2}  into \eqref{trans q1}, we obtain
	\begin{align*}
		Q&-e^{(b_1-b_2+\frac{\re z_1^2(a_1-a_2)}{2\im z_1^2})i}\psi^{z_1}(t,x+\frac{a_1-a_2}{2\im z_1^2})\\
		=&-\left( \frac{z_1e^{a_2-a_1}e^{2\re\theta}+\bar{z}_1e^{a_1-a_2}e^{-2\re\theta}}{z_1e^{a_2-a_1}e^{-2\re\theta}+\bar{z}_1e^{a_1-a_2}e^{2\re\theta}}\right) ^2\left( \dfrac{1+\bar{W}_2}{1+W_2}\right) ^2q^{(1)}\\
		&+e^{(b_1-b_2+\frac{\re z_1^2(a_1-a_2)}{2\im z_1^2})i}\psi^{z_1}(t,x+\frac{a_1-a_2}{2\im z_1^2})\left(1-\left( \dfrac{1+\bar{W}_2}{1+W_2}\right) ^2 \right) \\
		&+\left( \frac{z_1e^{a_2-a_1}e^{2\re\theta}+\bar{z}_1e^{a_1-a_2}e^{-2\re\theta}}{z_1e^{a_2-a_1}e^{-2\re\theta}+\bar{z}_1e^{a_1-a_2}e^{2\re\theta}}\right) ^2\frac{W_1-W_2\frac{e^{(b_1-b_2)i}e^{-2\im\theta i}}{z_1e^{a_2-a_1}e^{2\re\theta}+\bar{z}_1e^{a_1-a_2}e^{-2\re\theta}}}{1+W_2}.
	\end{align*}
Consequently, there exists a small $\delta>0$ such that when $\|q^{(1)}\|_2<\delta$,
	\begin{align*}
		&\|	Q-e^{(b_1-b_2+\frac{\re z_1^2(a_1-a_2)}{2\im z_1^2})i}\psi^{z_1}(t,x+\frac{a_1-a_2}{2\im z_1^2})\|_2\\
		&\leq\bigg\|\left( \frac{z_1e^{a_2-a_1}e^{2\re\theta}+\bar{z}_1e^{a_1-a_2}e^{-2\re\theta}}{z_1e^{a_2-a_1}e^{-2\re\theta}+\bar{z}_1e^{a_1-a_2}e^{2\re\theta}}\right) ^2\left( \dfrac{1+\bar{W}_2}{1+W_2}\right) ^2\bigg\|_\infty\|q^{(1)}\|_2\\
		&+\|\psi^{z_1}(t,\cdot)\|_2\big\|\left( \dfrac{1+\bar{W}_2}{1+W_2}\right) ^2-1 \big\|_\infty \\
		&+\bigg\| \frac{z_1e^{a_2-a_1}e^{2\re\theta}+\bar{z}_1e^{a_1-a_2}e^{-2\re\theta}}{z_1e^{a_2-a_1}e^{-2\re\theta}+\bar{z}_1e^{a_1-a_2}e^{2\re\theta}} \bigg\|_\infty^2\bigg\|\frac{W_1-W_2\frac{e^{(b_1-b_2)i}e^{-2\im\theta i}}{z_1e^{a_2-a_1}e^{2\re\theta}+\bar{z}_1e^{a_1-a_2}e^{-2\re\theta}}}{1+W_2}\bigg\|_2\\
		&\lesssim\|q^{(1)}\|_2.
	\end{align*}
	This completes the proof of Lemma \ref{lemma3.3}.
\end{proof}

\section{Proof  of stability theorem }\label{sec: proof}
We are now in a position to provide a proof of Theorem \ref{mainthm}.

\begin {proof} [Proof of Theorem \ref{mainthm}] The proof  is divided into two steps. Firstly, we establish the $L^2$-orbital stability of the DNLS equation \eqref{DNLS} under the condition $q_0\in H^3(\mathbb{R})$.

As demonstrated in Section \ref{sec 2}, the B\"acklund transformation maps an $L^2$-neighborhood of the 1-soliton solution to that of zero. Specifically, Lemma \ref{lemma1} indicates that if $|q_0-\psi^{z_0}_0|_2<\varepsilon$, then the spectral problem \eqref{lax0x} associated with $q_0$ has an eigenvalue $z_1$ and an eigenvector $\Phi:=\Phi(q_0,z_1)$ satisfying the normalization condition \eqref{normPhi}, with $|z_1-z_0|\lesssim |q_0-\psi^{z_0}_0|_2$. Furthermore,
	\begin{align*}
		\|q_0-\psi^{z_1}_0\|_2\leq \|q_0-\psi^{z_0}_0\|_2+\|\psi^{z_1}_0-\psi^{z_0}_0\|_2\lesssim\|q_0-\psi^{z_0}_0\|_2.
	\end{align*}
Let  $\varepsilon$ be sufficiently small, such that $\|q_0-\psi^{z_1}_0\|_2 < \epsilon$ required in Lemma \ref{lemma2.2}.
	It thus follows from  Lemma \ref{lemma2.2} that the B\"acklund transformation constructed by $\Phi$  maps $q_0$ to a small $q_0^{(1)}$ satisfying
	\begin{align*}
		\|q_0^{(1)}\|_2\lesssim\|q_0-\psi^{z_1}_0\|_2\lesssim\|q_0-\psi^{z_0}_0\|_2.
	\end{align*}
	It also maps $\Phi$ to  a vector solution $\Phi^{(1)}:=\Phi^{(1)}(q^{(1)}_0,z_1)=(\Phi^{(1)}_1,\Phi^{(1)}_2)^T$ in \eqref{Phi(1)} of the spectral problem (\ref{lax0x}) associated to $q^{(1)}_0$. Choose $\varepsilon$ to be sufficiently small such that $\|q_0^{(1)}\|_2  <  \delta$ with $\delta$ determined in Lemma \ref{lemma3.3}.
	
We  now focus on the time $t$. Given that the time evolution discussed in Section \ref{sec 3} requires the condition $q_0^{(1)}\in H^3(\mathbb{R})$, we initially assume that $q_0^{(1)}\in C([0,\infty),H^3(\mathbb{R}))$. Let $q^{(1)}\in H^3(\mathbb{R})$ denote the solution of the DNLS equation \eqref{DNLS} with initial value $q_0^{(1)}$. The Lax pair associated with $q^{(1)}$, having a solution under the spectrum $z_1$, yields a column vector function solution $\nu:=\nu(q^{(1)},z_1)$ as defined in \eqref{def: phi}:
	\begin{align}
		\nu(t,x)= e^{a_1(t)+b_1(t)i}e^{-iz_1^2x}\mu_{1}(t,x)+e^{a_2(t)+b_2(t)i}e^{iz_1^2x}\mu_{2}(t,x),\label{nv}
	\end{align}
	where the coefficients $a_j$, $b_j$, $j=1,2$, may rely on time $t$ and $\mu_1$, $\mu_2$ are constructed in Lemma \ref{lemma3.3}. 
	 Note that $\nu$ also admit the $t$-part of Lax pair \eqref{lax0t}. Then via substituting \eqref{nv} into \eqref{lax0t}, it transpires that $a_j$, $b_j$, $j=1,2$ are constants.
	Specially, when $t=0$, by the linear superposition principle, there exists two pairs of constants $(a_j,b_j)$, $j=1,2$, such that
	\begin{align*}
		\Phi^{(1)}(x)= e^{a_1+b_1i}e^{-iz_1^2x}\mu_{1}(0,x)+e^{a_2+b_2i}e^{iz_1^2x}\mu_{2}(0,x)=\nu(0,x),
	\end{align*}
for $\Phi^{(1)}(x)$ given in \eqref{Phi(1)}.
	Therefore, Lemma \ref{lemma3.3} gives that via the B\"acklund transformation \eqref{trans q1}, we obtain a new solution  $Q\in  C([0,\infty),H^3(\mathbb{R}))$ of the  DNLS equation \eqref{DNLS} with
\begin{align}
	\|Q-e^{(b_1-b_2+\frac{\re z_1^2(a_1-a_2)}{2\im z_1^2})i}\psi^{z_1}(t,x+\frac{a_1-a_2}{2\im z_1^2})\|_2\leq C(\delta)\|q^{(1)}\|_2.
\end{align}
	When $t=0$, from its definition one can check that $Q(0,x)=q_0(x)$.
	Indeed, from  the definitions of $q^{(1)}$, $\Phi^{(1)}$  and $Q$ in \eqref{defq01}, \eqref{Phi(1)} and \eqref{trans q1}, respectively, it appears that
	\begin{align*}
		Q(0,x)=\left(\frac{z_1|\Phi^{(1)}_2|^2+\bar{z}_1|\Phi^{(1)}_1|^2}{z_1|\Phi^{(1)}_1|^2+\bar{z}_1|\Phi^{(1)}_2|^2} \right)^2  &\left[ \frac{2i(z_1^2-\bar{z}_1^2)\Phi^{(1)}_1\bar{\Phi}^{(1)}_2}{z_1|\Phi^{(1)}_2|^2+\bar{z}_1|\Phi^{(1)}_1|^2}-q_0^{(1)}\right]\\
		=\left(\frac{z_1|\Phi_1|^2+\bar{z}_1|\Phi_2|^2}{z_1|\Phi_2|^2+\bar{z}_1|\Phi_1|^2} \right)^2&\left[\frac{2i(z_1^2-\bar{z}_1^2)\Phi_1\bar{\Phi}_2(z_1|\Phi_2|^2+\bar{z}_1|\Phi_1|^2)}{(z_1|\Phi_1|^2+\bar{z}_1|\Phi_2|^2)^2}\right.\\
		&\left.- \left(\frac{z_1|\Phi_2|^2+\bar{z}_1|\Phi_1|^2}{z_1|\Phi_1|^2+\bar{z}_1|\Phi_2|^2} \right)^2 \left[\frac{2i(z_1^2-\bar{z}_1^2)\Phi_1\bar{\Phi}_2}{z_1|\Phi_2|^2+\bar{z}_1|\Phi_1|^2}- q_0\right]\right]
		=q_0
	\end{align*}
	It thereby follows from the uniqueness of the global solution of the DNLS equation \eqref{DNLS} with initial value $q(0) =q_0$ that $Q\equiv q$.
Hence we have proved Theorem \ref{mainthm} in  $H^3(\mathbb{R})$, namely, 
\begin{align}
	\|q-e^{(b_1-b_2+\frac{\re z_1^2(a_1-a_2)}{2\im z_1^2})i}\psi^{z_1}(t,x+\frac{a_1-a_2}{2\im z_1^2})\|_2\lesssim\|q_0-\psi^{z_0}_0\|_2.
\end{align}
	
For  $q_0$ only belongs to $L^2(\mathbb{R})$, there exists   an approximating sequence $\{q_{0,n}\}_{n\in\mathbb{N}} \subset\mathcal{S}(\mathbb{R}) \subset H^3(\mathbb{R})$ converging to $q_0$ as $n\to\infty$ in the $L^2$-norm. Without loss of generality, we assume that for every $n\in\mathbb{N}$,
	\begin{align*}
		\|q_{0,n}-\psi^{z_0}_0\|_2\leq \|q_0-\psi^{z_0}_0\|_2.
	\end{align*}
Based on the above analysis, we can identify a sequence of eigenvalues ${z_n,\ }{n\in\mathbb{N}}$  such that $$|z_n-z_0|\lesssim \|q_{0,n}-\psi^{z_0}_0\|_2.$$ Let $q_n \in C([0,\infty),H^3(\mathbb{R}))$ denote the global solution of \eqref{DNLS} with the initial value $q_{0,n}$.
	Then for every $t>0$, there exists three sequences of constants $\{a_n\}$, $\{b_n\}$ and $\{z_n\}$ such that
	\begin{align}
	\|q_n(t,\cdot)-e^{bi}\psi^{z_n}(t,\cdot+a) \|_2\leq C\|q_{0,n}-\psi^{z_0}_0\|_2.\label{est:qn}
	\end{align}
On the other hand, there exists a  sub-sequence $\{n_k\}$  of $\mathbb{N}$ such that $\{z_{n_k}\}$ converges. We denote $\underset{k\to\infty}{\lim}z_{n_k}=z_1$. Note that  $q_n$ converges to $q$ uniformly for $t\in[0,T_0]$ for every $T_0>0$ \cite{KillipDNLS}.
 Namely, for  $\varepsilon=\|q_{0}-\psi^{z_0}_0\|_2$, there exists a $K_0>0$ such that  for any $t\in[0,T_0]$
 \begin{align*}
 	\|q_{n_{K_0}}(t,\cdot)-q(t,\cdot)\|_2\leq \|q_{0}-\psi^{z_0}_0\|_2.
 \end{align*}
It then follows that for any $t\in[0,T_0]$,
\begin{align*}
	\inf_{a,b\in\mathbb{R}}\|q(t,\cdot)-e^{bi}\psi^{z_1}(t,\cdot+a) \|_2\leq& \|q_{n_{K_0}}(t,\cdot)-e^{b_{K_0}i}\psi^{z_{n_{K_0}}}(t,\cdot+a_{K_0}) \|_2+\|q(t,\cdot)-q_{n_{K_0}}(t,\cdot)\|_2\\
	&+\inf_{a,b\in\mathbb{R}}\|e^{bi}\psi^{z_{n_{K_0}}}(t,\cdot+a)-\psi^{z_1}(t,\cdot)\|_2\\
	\lesssim &\|q_{0,{n_{K_0}}}-\psi^{z_0}_0\|_2+\|q(t,\cdot)-q_{n_{K_0}}(t,\cdot)\|_2+|z_{n_{K_0}}-z_1|\\
	\lesssim &\|q_{0}-\psi^{z_0}_0\|_2.
\end{align*}
As the right-hand side of the above estimate is independent of $T_0$, this generates the desired stability estimate in Theorem \ref{mainthm}.
\end {proof}


	\noindent\textbf{Acknowledgements}
	
	The work of Fan and Yang is partially  supported by  NSFC  under grants 12271104, 51879045  and China Postdoctoral Science Foundation. The work of Liu is partially supported by the Simons Foundation under grant 499875.

	\hspace*{\parindent}
	\\
	
\end{document}